\documentclass[11pt]{amsart}

\usepackage{amsfonts}
\usepackage{amscd}
\usepackage{amssymb}

\allowdisplaybreaks

\newtheorem{thm}{Theorem}[section]
\newtheorem{cor}[thm]{Corollary}
\newtheorem{lem}[thm]{Lemma}

\newtheorem{prop}[thm]{Proposition}
\theoremstyle{definition}

\theoremstyle{remark}
\newtheorem{rem}[thm]{Remark}

\numberwithin{equation}{section}

\newcommand{\R}{\mathbb R}

\newcommand{\Na}{\mathbb N}
\newcommand{\He}{\mathbb H}

\newcommand{\la}{\lambda}
\newcommand{\C}{{\mathbb C}}

\newcommand{\N}{\nabla }

\newcommand{\gm}{\gamma}

\renewcommand{\Re}{\operatorname{Re}}
\renewcommand{\Im}{\operatorname{Im}}
\newcommand{\tr}{\operatorname{tr}}

\usepackage{color}

\title[Hardy's inequality on the Heisenberg group]
{ Hardy's inequality for fractional powers of the\\ sublaplacian
on the Heisenberg group}

\author[L. Roncal and S. Thangavelu]{Luz Roncal  and Sundaram Thangavelu}

\address[L. Roncal]{Departamento de Matem\'aticas y Computaci\'on\\
    Universidad de La Rioja\\
    26006 Logro\~no, Spain}
\email{luz.roncal@unirioja.es}

\address[S. Thangavelu]{Department of Mathematics\\
 Indian Institute of Science\\
560 012 Bangalore, India}

\email{veluma@math.iisc.ernet.in}

\keywords{Hardy inequality, fractional order operator, sublaplacian, Heisenberg group, heat semigroup, fundamental solutions, uncertainty principle}
\subjclass[2010]{Primary: 43A80. Secondary: 26D15, 35A08, 46E35}
\thanks{The first author was
supported by the grant MTM2015-65888-C04-4-P from Government of Spain and the second author by J. C. Bose Fellowship (DSTO-TV-825) from DST, Government of India.}

\begin{document}

\maketitle

%------------------------------------------------------------------------------------
\begin{abstract}

We prove Hardy inequalities for the conformally invariant fractional powers of the sublaplacian on the Heisenberg group $\He^n$. We prove two versions of  such inequalities depending on whether the weights involved are  non-homogeneous or homogeneous. In the first case, the constant arising in the Hardy inequality turns out to be optimal. In order to get our results, we will use ground state representations. The key ingredients to obtain the latter are some explicit integral representations for the fractional powers of the sublaplacian and a generalized result by M. Cowling and U. Haagerup.
The approach to prove the integral representations is via the language of semigroups. As a consequence of the Hardy inequalities we also obtain versions of Heisenberg uncertainty inequality for the fractional sublaplacian.
\end{abstract}

%%%%%%%%%%%%%%%%%%%%%%%%%%%%%%%%%%%%%%%%%%%%%%%
\section{Introduction and main results}
%%%%%%%%%%%%%%%%%%%%%%%%%%%%%%%%%%%%%%%%%%%%%%%

The study and understanding of various kinds of weighted and unweighted inequalities for differential operators and the Fourier transform has been a matter of intensive research. This interest has been triggered and sustained by the importance of such inequalities in applications to problems in analysis, mathematical physics, spectral theory, fluid mechanics and stability of  matter. Moreover, the sharpness of the constants involved in these inequalities is the key in establishing existence and non-existence results for certain non-linear Schr\"odinger equations.

For instance, the Pitt's inequality, the Hardy--Littlewood--Sobolev inequality and the logarithmic Sobolev inequality are in connection with the measure of uncertainty \cite{B1,B2,B4}. The Sobolev, Hardy, or Hardy--Sobolev type inequalities are applied to prove stability of relativistic matter (see \cite{FLS}). They also deliver insight on the geometric structure of the space considered, and the knowledge of the best constants also help to solve isoperimetric inequalities or decide the existence of solutions of certain PDE's, see \cite{BFM} for a description of these topics.

A lot of work concerning these inequalities has been developed in the context of the Euclidean space and Riemannian manifolds, but not very much has been done in the framework of subRiemannian geometry, in particular in the Heisenberg group. We refer the remarkable work by R. L. Frank and E. H. Lieb \cite{FL} where they derive sharp constants for the Hardy--Littlewood--Sobolev inequalities on the Heisenberg group. We also refer the reader to \cite{Am2,B11,DGP,GL} concerning several kinds of inequalities related to either the Grushin operator, or in Carnot--Carath\'eodory spaces, or on the Heisenberg group. There is a vast  literature on this topic and our bibliography refers only to a very small fraction of the articles dealing with such inequalities and their applications.

In this article we are concerned with Hardy-type inequalities for the conformally invariant (or covariant, both nomenclatures seem to be used with the same meaning in the literature) fractional powers of the sublaplacian $ \mathcal{L} $ on the
Heisenberg group $ \He^n$. Some Hardy inequalities are already known for the sublaplacian, see for instance \cite{Am1, GL, BCG}, and also the very recent work by P. Ciatti, M. Cowling and F. Ricci \cite{CCR} (see Remark \ref{rem:CCR} below). However, in \cite{BCG} and \cite{CCR} where the fractional powers are treated, the authors have not paid attention to the sharpness of the constants.

To begin with, let us recall two inequalities in the case of the  Laplacian $ \Delta = -\sum_{j=1}^n \frac{\partial^2}{\partial x_j^2} $ on $\R^n$. First, the standard Sobolev embedding
$ W^{s/2,2}(\R^n) \hookrightarrow  L^{2n/(n-s)}(\R^n) $ for $ 0 < s < n $ leads to the optimal inequality
\begin{equation}
\label{eq:HLS}
\|f\|_q^2  \leq c_{n,s} \langle \Delta^{s/2}f, f\rangle
\end{equation}
with $ c_{n,s} = \omega_n^{-s/n} \frac{\Gamma(\frac{n-s}{2})}{\Gamma(\frac{n+s}{2})}$ where
$ \omega_n $ is the surface measure of the unit sphere $ \mathbb{S}^n $ in $ \R^{n+1}$ and $q=(2n)/(n-s)$. Here and later, the symbol $\langle \cdot,\cdot\rangle$ denotes the inner product in the corresponding space. The above inequality is usually referred to as the Hardy--Littlewood--Sobolev (HLS) inequality  for the fractional Laplacian $ \Delta^{s/2} $ in the literature.

Secondly, a Hardy-type inequality has the shape
\begin{equation}
\label{eq:hardy-type}
\int_{\R^n} \frac{|f(x)|^2}{(1+|x|^2)^s} \, dx \leq b_{n,s}  \langle \Delta^{s/2}f, f\rangle,
\end{equation}
for certain constant $b_{n,s}$. It is easy to see that a Hardy-type inequality can be obtained from the HLS inequality. Indeed, one observes that in view of  Holder's inequality applied to the left hand side of \eqref{eq:hardy-type} with $ q = (2n)/(n-s) $, it follows that
\begin{equation}
\label{eq:hardyfromHLS}
\int_{\R^n} \frac{|f(x)|^2}{(1+|x|^2)^s} \, dx \leq a_n^{s/n} \|f\|_q^2
\end{equation}
with $ a_n = \int_{\R^n} (1+|x|^2)^{-n} \, dx.$ Hence, in view of \eqref{eq:HLS} we immediately get the Hardy-type inequality with $ b_{n,s} = a_n^{s/n} c_{n,s}$.
In the case of HLS inequality it is known that the optimizers are given by dilations and translations of the function $ (1+|x|^2)^{-n/q},$  see e.g. \cite{BFM}. The constant in \eqref{eq:HLS} is sharp but not the one in the Hardy inequality \eqref{eq:hardyfromHLS}, obtained from the HLS.

There is another form of Hardy-type inequality where the function $ (1+|x|^2)^{-s} $ is replaced by the homogeneous potential $ |x|^{-s}$:
for $ 0 < s < n/2, f \in C_0^\infty(\R^n) $, this inequality reads as
\begin{equation}
\label{eq:HardyHomoEucl}
\int_{\R^n} \frac{|f(x)|^2}{|x|^{2s}} \, dx \leq C_{n,s}   \langle \Delta^{s}f, f\rangle
\end{equation}
where the sharp constant $C_{n,s} $ is given by
$$
C_{n,s} = 4^{-s} \frac{\Gamma(\frac{n-2s}{4})^2}{\Gamma(\frac{n+2s}{4})^2}.
$$
Inequality \eqref{eq:HardyHomoEucl} is a generalization of the original Hardy's inequality proved for the gradient of $ f$: for $ n \ge 3$,
$$ \frac{(n-2)^2}{4} \int_{\R^n} \frac{|f(x)|^2}{|x|^{2}} \, dx \leq  \int_{\R^n} |\N f(x)|^2 \, dx.$$
The sharp constant $ C_{n,s} $ was found in \cite{B1, Herbst, Yafaev}. It is also known that the equality is not obtained in the class of functions for which both sides of the inequality are finite. Later, Frank, Lieb, and R. Seiringer \cite{FLS} found a different proof of the inequality \eqref{eq:HardyHomoEucl} when $ 0 < s < \min\{1, n/2\} $ by using a \textit{ground state representation}, which enhanced the previous results.

In this work we prove analogues of Hardy-type inequalities for fractional powers of the sublaplacian $ \mathcal{L} $ on the Heisenberg group $ \He^n.$ Instead of considering powers of $ \mathcal{L} $ we will consider conformally invariant fractional powers $ \mathcal{L}_s $, see Subsection \ref{sub:fractionalSub} for definitions, and prove two versions of Hardy inequalities, one with a non-homogeneous and another with a homogeneous weight function. From the inequalities for $ \mathcal{L}_s $ we can deduce corresponding inequalities for $ \mathcal{L}^s $, as the operators $ \mathcal{L}^s \mathcal{L}_{-s} $ are bounded on $ L^2(\He^n).$

The conformally invariant fractional powers $ \mathcal{L}_s $ occur naturally in the context of CR geometry on the Heisenberg group and also on the sphere $ \mathbb{S}^{2n+1}$. We refer the works \cite{ BFM,BOO, FL,JW} for more information on these operators. They also arise in connection with the extension problem on the Heisenberg group as expounded in the recent work of Frank et al \cite{FGMT}.

We denote by $ W^{s,2}(\He^n) $ the Sobolev space consisting of all $ L^2 $ functions for which $ \mathcal{L}^{s/2} f \in L^2(\He^n)$. Therefore, $ W^{s,2}(\He^n) $ is a Sobolev space naturally associated to $ \mathcal{L}$. Note that an $ f \in L^2(\He^n) $ belongs to $ W^{s,2}(\He^n) $ if and only if $ \mathcal{L}_{s/2}f $ belongs to $ L^2(\He^n).$  We now state our first inequality for $ \mathcal{L}_s $ with a non-homogeneous weight function.

\begin{thm}[Hardy inequality in the non-homogeneous case]
\label{thm:HardynonH}
Let $0< s < \frac{n+1}{2}$ and $ \delta > 0$.  Then
$$
(4\delta)^s \frac{\Gamma\big(\frac{1+n+s}{2}\big)^2}
{\Gamma\big(\frac{1+n-s}{2}\big)^2}\int_{\He^n}\frac{|f(z,w)|^2}{\big( (\delta+\frac14|z|^2)^2+w^2\big)^{s}}\,dz\,dw \leq   \langle \mathcal{L}_sf,f\rangle
$$
for all functions $ f \in W^{s,2}(\He^n)$.
\end{thm}

The above inequality is optimal. In fact, the functions $ u_{-s,\delta} $ defined in \eqref{eq:usdelta} optimize the above inequality as will be checked later.

As in the Euclidean case studied by \cite{FLS},  we can get an expression for the error term in the above inequality when $ 0 < s < 1.$  Let
$$
\mathcal{H}_s[ f] :=  \langle \mathcal{L}_sf, f\rangle - C_{s,\delta} \int_{\He^n}  \frac{|f(z,w)|^2}{ \big((\delta+\frac14 |z|^2)^2+w^2\big)^s} dz\,dw
$$
where   $ C_{s,\delta} = (4\delta)^s \frac{\Gamma(\frac{n+1+s}{2})^2}{\Gamma(\frac{n+1-s}{2})^2}.$ Then we have the following result which is known as the ground state representation. In what follows the function $ u_{-s,\delta} $ is defined in \eqref{eq:usdelta}.

\begin{thm}[Ground state representation]
\label{thm:gsr}
Let $ 0 < s <1 $ and $ \delta >0.$ If $ f \in C_0^\infty(\He^n) $ and $ g(x) =  f(x)u_{-s,\delta}(x)^{-1} $ then
$$
\mathcal{H}_s[f]= a_{n,s}  \int_{\He^n}\int_{\He^n} \frac{|g(x)-g(y)|^2}{ |y^{-1}x|^{Q+2s}} \,u_{-s,\delta}(x) \,u_{-s,\delta}(y)\,dx\,dy,
$$
where $a_{n,s}$ is an explicit positive constant given by \eqref{eq:ansNONh}.
\end{thm}

\begin{rem}
It is possible to deduce a slightly weaker form of the inequality in Theorem \ref{thm:HardynonH} from the sharp HLS inequality proved recently by Frank and Lieb in \cite{FL}. This inequality, as stated in \cite[(3.2)]{BFM}, reads as
$$
\frac{\Gamma\big(\frac{1+n+s}{2}\big)^2}
{\Gamma\big(\frac{1+n-s}{2}\big)^2}  \omega_{2n+1}^{\frac{s}{n+1}}  \bigg(\int_{\He^n} |g(z,w)|^{\frac{2(n+1)}{n+1-s}} dz dw \bigg)^{\frac{n+1-s}{(n+1)}}
\leq \langle {L}_sg,g\rangle
$$
where $ L_s $ is the conformally covariant fractional power associated to a slightly different  sublaplacian  (see \cite{BFM}) adapted to a different  group structure. By applying Holder's inequality  we can prove
$$
 4^s \frac{\Gamma\big(\frac{1+n+s}{2}\big)^2}
{\Gamma\big(\frac{1+n-s}{2}\big)^2}\int_{\He^n}\frac{|g(z,w)|^2}{\big( (1+|z|^2)^2+w^2\big)^{s}}\,dz\,dw \leq 2^{\frac{s}{n+1}} \langle L_sg,g\rangle.
$$
Consequently, rewriting the above in terms of our sublaplacian, we have the inequality
$$
 4^s \frac{\Gamma\big(\frac{1+n+s}{2}\big)^2}
{\Gamma\big(\frac{1+n-s}{2}\big)^2}\int_{\He^n}\frac{|f(z,w)|^2}{\big( (1+\frac14|z|^2)^2+w^2\big)^{s}}\,dz\,dw \leq 2^{\frac{s}{n+1}}  \langle \mathcal{L}_sf,f\rangle
$$
which is weaker than the inequality stated in Theorem \ref{thm:HardynonH}. We refer to Section \ref{sub:HLS} for details.
\end{rem}

From Theorem \ref{thm:HardynonH} we can deduce a Hardy inequality for the pure fractional power $ \mathcal{L}^s.$  It can be shown that the operator $ U_s := \mathcal{L}_s \mathcal{L}^{-s} $ is bounded and its operator norm is given by the constant
 \begin{equation}
 \label{eq:Us}
 \|U_s\| = \sup_{ k \ge 0}  \bigg(\frac{2k+n}{2}\bigg)^{-s} \frac{\Gamma(\frac{2k+n}{2}+\frac{1+s}{2})} {\Gamma(\frac{2k+n}{2}+\frac{1-s}{2})}.
 \end{equation}
Using an integral representation for a ratio of gamma functions,  $ \|U_s\| $ can be estimated, see Subsection \ref{sub:Ls}. The Hardy inequality for $\mathcal{L}^s$ is shown in the following theorem.

\begin{thm}
\label{thm:2}
Let $0< s < \frac{n+1}{2}$  and $ \delta > 0$.  Then
$$
(4\delta)^s \frac{\Gamma\big(\frac{1+n+s}{2}\big)^2}
{\Gamma\big(\frac{1+n-s}{2}\big)^2}\int_{\He^n}\frac{|f(z,w)|^2}{\big( (\delta+\frac14|z|^2)^2+w^2\big)^{s}}\,dz\,dw \leq  \|U_s\|  \langle \mathcal{L}^sf,f\rangle
$$
for all functions $ f \in W^{s,2}(\He^n)$.
\end{thm}

We now turn our attention to a version of Hardy inequality with an homogeneous weight function.  As before, we do not deal directly with $ \mathcal{L}^s $ and the required inequality will be proved  from the following inequality for the related operator
\begin{equation}
\label{eq:Lambda}
 \Lambda_s := \mathcal{L}_{1-s}^{-1}\mathcal{L},
\end{equation}
 which behaves like $ \mathcal{L}^s$.
\begin{thm}[Hardy inequality in the homogeneous case]
\label{thm:HardyH}
Let $0 < s <1. $ Then
$$
 \frac{2^{2n+3s}\Gamma\big(\frac{n+s}{2}\big)^2}
{\Gamma(1-s)\Gamma\big(\frac{n}{2}\big)^2}\int_{\He^n}\frac{|f(z,w)|^2}{|(z,w)|^{2s}}\,dz\,dw \leq \langle \Lambda_sf,f\rangle
$$
for all $ f \in C_0^\infty(\He^n) $.
\end{thm}
We also have a ground state representation in this case, see Theorem \ref{thm:gsrhomo}.

As the operator $ V_s := \Lambda_s \mathcal{L}^{-s} = \mathcal{L}_{1-s}^{-1}\mathcal{L}\mathcal{L}^{-s} $ is bounded on $ L^2(\He^n) $ we can immediately get the following result.

\begin{thm}
\label{thm:4}
Let $0 < s <1. $ Then
$$
 \frac{2^{2n+3s}\Gamma\big(\frac{n+s}{2}\big)^2}
{\Gamma(1-s)\Gamma\big(\frac{n}{2}\big)^2}\int_{\He^n}\frac{|f(z,w)|^2}{|(z,w)|^{2s}}\,dz\,dw \leq \|V_s\|  \langle \mathcal{L}^sf,f\rangle
$$
for all $ f \in C_0^\infty(\He^n) $.
\end{thm}
We will show an estimate for $\|V_s\|$ in Subsection \ref{sub:Ls}.

We do not know if the constants appearing in Theorems \ref{thm:HardyH} and \ref{thm:4} are optimal or not.  We also remark that it is not possible to obtain the homogeneous case from the non homogeneous just by letting $ \delta $ go to $ 0$.

\begin{rem}
\label{rem:CCR}
An analogue of Theorem \ref{thm:4} in the more general context of stratified groups has been proved recently in the nice work \cite{CCR} using different methods. They have also deduced Heisenberg uncertainty principle and logarithmic uncertainty inequality for fractional powers of the sublaplacian. They do not have information about the constants involved.
\end{rem}

We can deduce Heisenberg type uncertainty inequalities for $\mathcal{L}_s$ and $\Lambda_s$ from our Hardy inequalities as well. This was done by N. Garofalo and E. Lanconelli for the sublaplacian in \cite[Corollary 2.2]{GL}.

\begin{cor}[Uncertainty principles for the fractional powers of the sublaplacian]
For all functions $f\in W^{s,2}(\He^n)$, we have
$$
(4\delta)^{s} \frac{\Gamma\big(\frac{1+n+s}{2}\big)^2}
{\Gamma\big(\frac{1+n-s}{2}\big)^2} \Big(\int_{\He^n}|f(z,w)|^2\,dz\,dw \Big)^2 \le \Big(\int_{\He^n}|f(z,w)|^2\big( (\delta+\frac14|z|^2)^2+w^2\big)^{s}\,dz\,dw \Big) \langle \mathcal{L}_sf,f\rangle
$$
provided $ 0 < s < \frac{n+1}{2}. $ In the smaller range $ 0 < s < 1 $ we have
$$
\frac{2^{2n+3s}\Gamma\big(\frac{n+s}{2}\big)^2}
{\Gamma(1-s)\Gamma\big(\frac{n}{2}\big)^2} \Big(\int_{\He^n}|f(z,w)|^2\,dz\,dw \Big)^2\le \Big(\int_{\He^n}|f(z,w)|^2|(z,w)|^{2s}\,dz\,dw\Big) \langle \Lambda_sf,f\rangle.
$$
\end{cor}

The uncertainty principles are obtained from the Hardy inequalities in Theorems \ref{thm:HardynonH} and \ref{thm:HardyH}. Indeed, if we denote by $\mathcal{W}(z,w)$ either the non-homogeneous weight $\big( (\delta+\frac14|z|^2)^2+w^2\big)^{s}$ or the homogeneous weight $|(z,w)|^{2s}$ we have, by Cauchy--Schwarz inequality,
$$
\int_{\He^n} |f(z,w)|^2 \,dz\,dw
\leq  \Big(\int_{\He^n} |f(z,w)|^2 \mathcal{W}(z,w) \,dz\,dw\Big)^{1/2} \Big(\int_{\He^n} |f(z,w)|^2 \mathcal{W}(z,w)^{-1}
\,dz\,dw\Big)^{1/2}.
$$
The last integral is bounded by $\langle \mathcal{L}_sf,f\rangle^{1/2}$ or $\langle \Lambda_sf,f\rangle^{1/2}$ times the corresponding constant, by Hardy's inequality.

Our results are based on ideas presented in \cite{FLS}. In this regard, we prove ground state representations for the fractional differential operators involved. The first goal to establish the ground state representations is the choice of the ``ground states'', which are intimately related to the fundamental solutions of the operators involved. To determine these ground states we use a result by M. Cowling and U. Haagerup, that we show here in a more general version, and with a slightly different proof. The other key ingredients are the integral representations with explicit kernels that we obtain for $\mathcal{L}_s$ and $\Lambda_s$. These integral representations seem to be new in the literature, and the approach we use to prove them is through the language of semigroups.

At this point, we would like to highlight the usefulness of the semigroup theory, that gives us the chance to get integral representations for our operators. Actually, the integral representation of the operator $\Delta^s$ in the Euclidean case, given for instance in \cite[Lemma 3.1]{FLS}, can be easily obtained with the semigroup approach, see Appendix.

As far as we know, apart from the results in \cite{CCR}, there is no work related to Hardy-type inequalities for fractional powers of the sublaplacian.  However, there are a couple of papers dealing with Hardy type inequalities involving the Heisenberg gradient. In \cite{AS}  Adimurthi and A. Sekar  have  proved the following  inequality for the Heisenberg gradient:
$$
\bigg(\frac{2(n+2-p)}{p}\bigg)^p \int_{\He^n} \frac{|z|^2 |f(z,w)|^p}{(|z|^4+w^2)^{\frac{p}{2}}} dz\,dw \leq \int_{\He^n} \frac{|\N_{\He}f(z,w)|^p}{|(z,w)|^{p-2}} dz\, dw
$$
valid for $ 1 < p < (n+2)$.  Observe that when $ p =2 $ the above inequality  is comparable to our result  with $ s = 1/2 $ but the weight functions are different though of the same homogeneity. Their proof relied on explicit computations of the gradient of the fundamental solution associated to the sublaplacian. A similar inequality with Carnot--Carath\'eodory distance in place of the homogeneous norm is proved by D. Danielli et al in \cite{DGP} but again only for the gradient.

Finally, we remark that though we treat only the Heisenberg group in this paper, all the results can be proved in the more general setting of $H$-type groups.

The outline of the paper is the following. In Section \ref{sec:prelim} we give preliminaries, definitions and facts concerning the Heisenberg group, the fractional powers of the sublaplacian, and the heat and certain modified heat kernels related to the sublaplacian. Next, in Section \ref{sec:fundamental}, we prove a slightly more general version of some results of Cowling--Haagerup in \cite[Section 3]{CH} which allows us to take the suitable weights involved in the Hardy inequalities. The integral representations for the operators $\mathcal{L}_s$ and $\Lambda_s$ are contained in Section \ref{sec:integralRep}. The ground state representations and the proofs of the Hardy inequalities stated as the main theorems are shown in Section \ref{sec:groundHardy}. In Section \ref{sec:groundHardy} we also compare the Hardy inequalities we have just obtained for the operators $\mathcal{L}_s$ and $\Lambda_s$ to the Hardy inequalities for the pure fractional powers $\mathcal{L}^s$. Moreover, we show with detail the weaker Hardy inequality that can be obtained from the HLS inequality in \cite{FL}. In the final Appendix we show an integral representation for the fractional powers of the Euclidean Laplacian by means of the semigroup and the Hardy inequality that is deduced from that.

%%%%%%%%%%%%%%%%%%%%%%%%%%%%%%%%%%%%%%%%%%%%%%%%
\section{Preliminaries on the Heisenberg group}
\label{sec:prelim}
%%%%%%%%%%%%%%%%%%%%%%%%%%%%%%%%%%%%%%%%%%%%%%%%

%%%%%%%%%%%%%%%%%%%%%%%%%%%%%%%%%%%%%%%%%%%%%%%%
\subsection{Representations of the Heisenberg group, Fourier and Weyl transforms}
\label{sec:Heisenberg}
%%%%%%%%%%%%%%%%%%%%%%%%%%%%%%%%%%%%%%%%%%%%%%%%

Let us first introduce some definitions and set up notations concerning the Heisenberg group. We refer the reader to the books of G. B. Folland \cite{Fo}, M. E. Taylor \cite{MT}, and the monograph \cite{BCT} of C. Berenstein et al. However, we  closely follow the notations used in \cite{STH}.  It is possible to work with Bargmann-Fock representations as was done in the papers by  \cite{CH} and others. Nevertheless, it will be enough to stick to the Schr\"odinger picture for our purposes.  We also warn the reader that our notation and certain definitions may be slightly different from those used by others.

Let $\He^n=\C^n\times \R$ denote the $(2n+1)$ dimensional Heisenberg group with the group law
$$
(z,w)(z',w')=\Big(z+z',w+w'+\frac12\Im(z\cdot \bar{z'})\Big),
$$
where $z,z'\in \C^n$ and $w,w'\in \R$.
We now recall some basic facts from the representation theory of the Heisenberg group. For each $\lambda\in \R^* = \R\setminus\{0\}$, we have an irreducible  unitary representation $\pi_{\lambda}$ of $\He^n$ realized on $ L^2(\R^n).$ The action of $ \pi_\lambda(z,w) $ on $ L^2(\R^n) $ is explicitly given by
$$
\pi_\lambda(z,w)\varphi(\xi) = e^{i\lambda w} e^{i(x\cdot \xi+\frac12 x\cdot y)}\varphi(\xi+y) $$
where $ \varphi \in L^2(\R^n) $ and $ z = x+iy.$ By a theorem of Stone and Von Neumann, any irreducible unitary representation of $ \He^n $ which acts as
$ e^{i\lambda w} \operatorname{Id} $ at the center of the Heisenberg group is unitarily equivalent to $ \pi_\lambda .$  In view of this, there are representations of $ \He^n $ which are realized on the Fock spaces and equivalent to $ \pi_\lambda.$ As we mentioned at the beginning, we will not use these representations and refer the reader to \cite{Fo} for details. There are also certain families of one dimensional representations which do not concern us here.

The group Fourier transform of a function  $f\in L^1(\He^n)$ is the operator-valued function defined, for each $\lambda\in \R^*$, by
\begin{equation*}
%\label{eq:groupFT}
\widehat{f}(\lambda) := \pi_{\lambda}(f)= \int_{\He^n}f(z,w)\pi_{\lambda}(z,w)\,dz\,dw.
\end{equation*}
With an abuse of language, we will call the group Fourier transform just the Fourier transform.
Observe that for each $\lambda$, $\widehat{f}(\lambda)$ is an operator acting on $L^2(\R^n)$. When $f\in L^1\cap L^2(\He^n)$, it can be shown that $\widehat{f}(\lambda)$ is a Hilbert--Schmidt operator and the Plancherel theorem holds:
\begin{equation}
\label{eq:Plancherel}
\int_{\He^n}|f(z,w)|^2\,dz\,dw=\frac{2^{n-1}}{\pi^{n+1}}
\int_{-\infty}^{\infty}\|\widehat{f}(\lambda)\|_{\operatorname{HS}}^2|\lambda|^n\,d\lambda,
\end{equation}
where $\|\cdot\|_{\operatorname{HS}}$ is the Hilbert--Schmidt norm given by $\|T\|^2_{\operatorname{HS}}=\tr(T^* T)$, for $T$ a bounded operator, being $T^*$ the adjoint operator of $T$.
By polarizing the Plancherel identity we get the Parseval formula:
$$
\int_{\He^n} f(z,w) \overline{g(z,w)} dz dw= \frac{2^{n-1}}{\pi^{n+1}}
\int_{-\infty}^{\infty}  \tr(\widehat{f}(\lambda)\widehat{g}(\lambda)^*) |\lambda|^n\,d\lambda.
$$

Let $ f^\lambda $ stand for the inverse Fourier transform of $ f $ in the \textit{central variable} $w$
\begin{equation}
\label{eq:inverseFT}
f^\lambda(z) = \int_{-\infty}^\infty f(z,w) e^{i\lambda w} dw.
\end{equation}
By taking the Euclidean Fourier transform of $ f^\lambda(z)$  in the variable $\lambda$, we obtain
\begin{equation}
\label{eq:lambdaFT}
f(z,w)=\frac{1}{2\pi}\int_{-\infty}^{\infty}e^{-i\lambda w}f^{\lambda}(z)\,d\lambda.
\end{equation}
We will use this formula quite often.
By the definition of $ \pi_\lambda(z,w) $ and $ \widehat{f}(\lambda) $ it is easy to see that
\begin{equation}
\label{eq:FTbis}
\widehat{f}(\lambda) = \int_{\C^n} f^\lambda(z) \pi_\lambda(z,0) dz.
\end{equation}
The operator which takes  a function $ g $ on $ \C^n $ into the operator
$$ \int_{\C^n} g(z) \pi_\lambda(z,0) dz $$ is called the Weyl transform of $ g $ and is denoted by $ W_\lambda(g)$. Thus $ \widehat{f}(\lambda) = W_\lambda(f^\lambda)$.

Taking the inverse Fourier transform in the central variable \eqref{eq:inverseFT} is an important tool which is quite often employed in studying problems on $ \He^n$. It also converts the group convolution on $ \He^n $ into the so-called twisted convolution on $ \C^n$. Let us recall that the convolution of $ f $ with $ g $ on $\He^n$ is defined by
$$
f*g(x) = \int_{\He^n} f(xy^{-1})g(y) dy, \quad x,y\in \He^n.
$$
 With  $x=(z,w)$ and $y=(z',w')$ the above takes the form
$$
f*g(z,w) = \int_{\He^n} f((z,w)(-z',-w'))g(z',w') dz' dw'
$$
from which a simple computation shows that
$$
(f*g)^\lambda(z) = \int_{\C^n} f^\lambda(z-z')g^\lambda(z') e^{\frac{i}{2}\Im(z\cdot \bar{z'})} dz'.
$$
The convolution appearing on the right hand side is called the $ \lambda$-twisted convolution and is denoted by $ f^\lambda*_\lambda g^\lambda(z).$
We remark that the relation $ \widehat{f*g}(\lambda) = \widehat{f}(\lambda) \widehat{g}(\lambda)$ yields, from the definitions above, the relation $ W_\lambda(f^\lambda*_\lambda g^\lambda)
=W_\lambda(f^\lambda)W_\lambda(g^\lambda).$

%%%%%%%%%%%%%%%%%%%%%%%%%%%%%%%%%%%%%%%%%%%%%%%%
\subsection{Hermite functions and the Heisenberg group}
\label{sec:Hermite}
%%%%%%%%%%%%%%%%%%%%%%%%%%%%%%%%%%%%%%%%%%%%%%%%

Now, for $\lambda\in \R^*$ and each $ \alpha \in \Na^n $, we introduce the family of Hermite functions
$$
\Phi_\alpha^\lambda(x) = |\lambda|^{\frac{n}{4}}\Phi_\alpha(\sqrt{|\lambda|}x), \quad x\in \R^n.
$$
Here, $\Phi_\alpha $ is the normalized Hermite function on $\R^n$ which is an eigenfunction of the Hermite operator $H = -\Delta+|x|^2 $ with eigenvalue $ (2|\alpha|+n)$, see for instance \cite[Chapter 1.4]{STH}. The system is an orthonormal basis for $L^2(\R^n) $. In terms of $ \Phi_\alpha^\lambda $ we have the following formula
$$
\|\widehat{f}(\lambda)\|_{\operatorname{HS}}^2=\sum_{\alpha\in \Na^n}\|\widehat{f}(\lambda) \Phi_\alpha^\lambda \|_{2}^2
$$
and hence, by \eqref{eq:Plancherel}, the Plancherel formula takes the form
$$
\int_{\He^n}|f(z,w)|^2\,dz\,dw=\frac{2^{n-1}}{\pi^{n+1}}
\int_{-\infty}^{\infty}\Big(\sum_{\alpha\in \Na^n}\|\widehat{f}(\lambda)\Phi_\alpha^\lambda\|_{2}^2\Big)|\lambda|^n\,d\lambda.
$$
Moreover, we can write the spectral decomposition of the scaled Hermite operator $ H(\lambda) =  -\Delta+|\lambda|^2 |x|^2 $ as
\begin{equation}
\label{eq:spectralHermite}
H(\lambda) = \sum_{k=0}^\infty (2k+n)|\lambda| P_k(\lambda),
\end{equation}
for $\lambda \in \R^*$, where $ P_k(\lambda) $ are the (finite-dimensional) orthogonal projections defined on $  L^2(\R^n)$ by
$$
P_k(\lambda)\varphi = \sum_{|\alpha| =k} (\varphi,\Phi_\alpha^\lambda) \Phi_\alpha^\lambda,
$$
where $ \varphi \in L^2(\R^n) $ and $(\cdot, \cdot)$ is the inner product in $L^2(\R^n)$.

On the other hand, we define the scaled Laguerre functions of type $ (n-1)$
\begin{equation}
\label{eq:Laguerre}
\varphi_k^\lambda(z) = L_k^{n-1}\Big(\frac12 |\lambda||z|^2\Big)e^{-\frac14 |\lambda||z|^2}.
\end{equation}
Here $ L_k^{n-1} $ are the Laguerre polynomials of type $ (n-1)$, see \cite[Chapter 1.4] {STH} for the definition and properties. It happens that $\{\varphi_k^\lambda\}_{k=0}^{\infty}$ forms an orthogonal basis for the subspace consisting of radial functions in $ L^2(\C^n).$  These functions play an important role in the analysis on the Heisenberg group. Indeed, the so-called special Hermite expansion of a function
$ g $ defined on $ \C^n $ written in its compact form reads as
$$
g(z)=(2\pi)^{-n} |\lambda|^n \sum_{k=0}^\infty  g*_\lambda \varphi_k^\lambda(z).
$$
The name \textit{special Hermite expansion} is due to the fact that the above is a compact form of the expansion in terms of the special Hermite functions $ (\pi_\lambda(z,0)\Phi_\alpha^\lambda,\Phi_\beta^\lambda) $ which are eigenfunctions of the Hermite operator on $ \C^n,$  see \cite{STU}. The connection betweeen the Hermite projections $P_k(\lambda)$ and the Laguerre functions $\varphi_k^\lambda$, via the Weyl transform, is given by the following important formula
\begin{equation}
\label{eq:WeylLaguerre}
W_\lambda(\varphi_k^\lambda) = (2\pi)^n |\lambda|^{-n} P_k(\lambda).
\end{equation}
Observe that, in particular, for any function $f$ on $\He^n$, we have the expansion
\begin{equation}
\label{eq:expansion}
f^{\lambda}(z)=(2\pi)^{-n} |\lambda|^n \sum_{k=0}^\infty  f^{\lambda}*_\lambda \varphi_k^\lambda(z).
\end{equation}
We remark that when $ f $ is radial in the $ z $ variable, i.e. $f(z,w)=f(r,w)$, $r=|z|$, its Fourier transform $ \widehat{f}(\lambda) $ becomes a function of the Hermite operator $ H(\lambda)$. To see this, it can be proved that
$$   f^{\lambda}*_\lambda \varphi_k^\lambda(z) = c_k^\lambda(f^\lambda) \varphi_k^\lambda(z) $$
where $ c_k^\lambda(f^\lambda) $ are the Laguerre coefficients of the radial function $ f^\lambda $ on $ \C^n$ given by
$$
c_k^\lambda(f^\lambda)= c_{n,\lambda}  \frac{k! (n-1)!}{(k+n-1)!} \int_{\C^n} f^\lambda(z) \varphi_k^\lambda(z) dz,
$$
where $c_{n,\lambda}$ is certain normalizing constant.
 Thus we have the expansion
$$
f^{\lambda}(z)=(2\pi)^{-n} |\lambda|^n \sum_{k=0}^\infty  c_k^\lambda(f^{\lambda}) \varphi_k^\lambda(z).
$$
Then, by taking the Weyl transform and making use of \eqref{eq:WeylLaguerre} we obtain
$$
\widehat{f}(\lambda) = \sum_{k=0}^\infty  c_k^\lambda(f^{\lambda}) P_k(\lambda).
$$
We will use these relations in the sequel, and refer the reader to \cite{STH} or \cite{STU} for more details.

%%%%%%%%%%%%%%%%%%%%%%%%%%%
\subsection{Fractional powers of the sublaplacian}
\label{sub:fractionalSub}
%%%%%%%%%%%%%%%%%%%%%%%%%%%

We begin with the definition of the sublaplacian on the Heisenberg group.
The Lie algebra of the Heisenberg group is generated by the $(2n+1)$ left invariant vector fields
$$
X_j=\bigg(\frac{\partial}{\partial x_j}+\frac12 y_j\frac{\partial}{\partial w}\bigg), \quad Y_j=\bigg(\frac{\partial}{\partial y_j}-\frac12 x_j\frac{\partial}{\partial w}\bigg), \quad T=\frac{\partial}{\partial w}, \quad j=1,2,\ldots,n.
$$
The sublaplacian $\mathcal{L}$ is defined by
$$
\mathcal{L}=- \sum_{j=1}^n(X_j^2+Y_j^2)
$$
which can be explicitly calculated. In fact, if we let
$$
N = \sum_{j=1}^n \Big(x_j \frac{\partial}{\partial y_j}-y_j \frac{\partial}{\partial x_j}\Big)
$$
then
$$
\mathcal{L} = -\Delta-\frac14 |z|^2 \frac{\partial^2}{\partial w^2}+N\frac{\partial}{\partial w}
$$
where $ \Delta$ is the Laplacian on $ \C^n$. This operator is the counterpart of the Laplacian on $\R^n$. Moreover, it is a second order subelliptic operator on $ \He^n $ which is homogeneous of degree two under the non-isotropic dilations $ \delta_r(z,w) = (rz,r^2w).$ A fundamental solution of $\mathcal{L}$ was found by Folland \cite{F}.

We proceed to obtain the spectral decomposition of the sublaplacian which will be then used to define fractional powers of $ \mathcal{L}.$ The decomposition is achieved via the special Hermite expansion introduced in the previous subsection.  The action of the Fourier transform on functions of the form $\mathcal{L}f$ and $Tf$ are given by
$$
(\mathcal{L}f)^{\widehat{}}(\lambda)
= \widehat{f}(\lambda)H(\lambda), \qquad (Tf)^{\widehat{}}(\lambda)=-i\lambda\widehat{f}(\lambda).
$$
If $ L_\lambda $ is the operator defined by the relation $ (\mathcal{L}f)^\lambda = L_\lambda f^\lambda $ then it follows that
$$  W_\lambda(L_\lambda f^\lambda) =  W_\lambda(f^\lambda) H(\lambda). $$
Recalling the spectral decomposition of $ H(\lambda) $ given in \eqref{eq:spectralHermite} and the identity \eqref{eq:WeylLaguerre} we obtain
$$   L_\lambda f^\lambda(z) = (2\pi)^{-n} \sum_{k=0}^\infty (2k+n)|\lambda|  f^\lambda*_\lambda \varphi_k^\lambda(z).$$
Thus, by taking the Fourier transform in the variable $\lambda$ \eqref{eq:lambdaFT}, the spectral decomposition of the sublaplacian is given by
\begin{equation}
\label{eq:spectral}
\mathcal{L}f(z,w) = (2\pi)^{-n-1} \int_{-\infty}^\infty \Big( \sum_{k=0}^\infty (2k+n)|\lambda|  f^\lambda*_\lambda \varphi_k^\lambda(z)\Big) e^{-i\lambda w}  |\lambda|^n d\lambda.
\end{equation}

Therefore, a natural way to define fractional powers of the sublaplacian is via the spectral decomposition
$$ \mathcal{L}^sf(z,w) = (2\pi)^{-n-1} \int_{-\infty}^\infty \Big( \sum_{k=0}^\infty \big((2k+n)|\lambda|\big)^s  f^\lambda*_\lambda \varphi_k^\lambda(z)\Big) e^{-i\lambda w}  |\lambda|^n d\lambda.$$
Note that $ \widehat{(\mathcal{L}^sf)}(\lambda) = \widehat{f}(\lambda)H(\lambda)^s.$

However, it is convenient to work with the following modified fractional powers $\mathcal{L}_s$. As mentioned in the introduction, the operators $ \mathcal{L}_s $ occur naturally in the context of CR geometry and scattering theory on the Heisenberg group. When we identify $ \He^n$ as the boundary of the Siegel's upper half space in $ \C^{n+1} $ the operators $ \mathcal{L}_s $ have the important property of being conformally invariant. For $0\le s < (n+1) $ the operator  $\mathcal{L}_s$ is defined by
\begin{equation}
\label{eq:modifFract1}
\mathcal{L}_sf(z,w) = (2\pi)^{-n-1} \int_{-\infty}^\infty \Big( \sum_{k=0}^\infty (2|\lambda|)^s \frac{\Gamma(\frac{2k+n}{2}+\frac{1+s}{2})}{ \Gamma(\frac{2k+n}{2}+\frac{1-s}{2})} f^\lambda*_\lambda \varphi_k^\lambda(z)\Big) e^{-i\lambda w}|\lambda|^n d\lambda.
\end{equation}
In short, the above means that $ \mathcal{L}_s $ is the operator (see \cite[(1.33)]{BFM})
\begin{equation*}
%\label{eq:modifFract2}
\mathcal{L}_s:=(2|T|)^s\frac{\Gamma\big(\frac{\mathcal{L}}{2|T|}+\frac{1+s}{2}\big)}
{\Gamma\big(\frac{\mathcal{L}}{2|T|}+\frac{1-s}{2}\big)}.
\end{equation*}
Thus $\mathcal{L}_s$ corresponds to  the spectral multiplier
\begin{equation}
\label{eq:multiplier}
(2|\lambda|)^s\frac{\Gamma\big(\frac{2k+n}{2}+\frac{1+s}{2}\big)}
{\Gamma\big(\frac{2k+n}{2}+\frac{1-s}{2}\big)}, \quad k\in \Na.
\end{equation}
Note that $ \mathcal{L}_1 = \mathcal{L} $ whose  explicit fundamental solution was found by Folland and given by a constant multiple of $ |(z,w)|^{-Q+2} $ where $ Q = 2(n+1) $ is the homogeneous dimension of $ \He^n.$ It is known that $ \mathcal{L}_s$ also has an explicit fundamental solution, see e.g. page 530 in \cite{CH} (more details will be given in  Section \ref{sec:fundamental}). This makes it more suitable than $ \mathcal{L}^s $, whose fundamental solution cannot be written down explicitly. Moreover,  $ \mathcal{L}_s $ is not very different from $ \mathcal{L}^s$. Using Stirling's formula for the Gamma function, it is easy to see that $ \mathcal{L}_s = U_s \mathcal{L}^{s} $ where $ U_s $ is a bounded operator on $ L^2(\He^n)$, as explained in the introduction.

In view of \eqref{eq:spectral}, by taking the inverse Fourier transform in the central variable, the operator $\mathcal{L}$ can be written as
\begin{equation}
\label{eq:Linverse}
\int_{-\infty}^{\infty}\mathcal{L}f(z,w)e^{i\lambda w}\,dw=(2\pi)^{-n}|\lambda|^n\sum_{k=0}^{\infty}(2k+n)|\lambda|
f^{\lambda}\ast_{\lambda}\varphi_k^{\lambda}(z).
\end{equation}
Analogously, in view of \eqref{eq:modifFract1}, by taking the inverse Fourier transform in the central variable, the operator $\mathcal{L}_s$ is given by the spectral decomposition
\begin{equation}
\label{eq:LsSpec}
\int_{-\infty}^{\infty}\mathcal{L}_sf(z,w)e^{i\lambda w}\,dw=(2\pi)^{-n}|\lambda|^n\sum_{k=0}^{\infty}(2|\lambda|)^{s}
\frac{\Gamma\big(\frac{2k+n}{2}+\frac{1+s}{2}\big)}
{\Gamma\big(\frac{2k+n}{2}+\frac{1-s}{2}\big)}f^{\lambda}\ast_{\lambda}\varphi_k^{\lambda}(z),
\end{equation}
and the inverse of the operator $\mathcal{L}_s$ is given by
\begin{equation}
\label{eq:Lsinverse}
\int_{-\infty}^{\infty}\mathcal{L}_s^{-1}f(z,w)e^{i\lambda w}\,dw=(2\pi)^{-n}|\lambda|^n\sum_{k=0}^{\infty}(2|\lambda|)^{-s}
\frac{\Gamma\big(\frac{2k+n}{2}+\frac{1-s}{2}\big)}
{\Gamma\big(\frac{2k+n}{2}+\frac{1+s}{2}\big)}f^{\lambda}\ast_{\lambda}\varphi_k^{\lambda}(z).
\end{equation}
Notice that  $\mathcal{L}_s^{-1} = \mathcal{L}_{-s} $, and it can be expressed by convolution with a kernel which we will explicitly calculate in Section \ref{sec:integralRep}.

%%%%%%%%%%%%%%%%%%%%%%%%%%%%%%%%%%%%%%%%%%%%%%%%
\subsection{Heat kernel and modified heat kernels for the sublaplacian}
%%%%%%%%%%%%%%%%%%%%%%%%%%%%%%%%%%%%%%%%%%%%%%%%

The sublaplacian is a self-adjoint, non-negative, hypoelliptic operator, and it generates a contraction semigroup which we denote by $ e^{-t\mathcal{L}}.$  This semigroup is defined by the relation $$
\widehat{(e^{-t\mathcal{L}}f)}(\lambda)=\widehat{f}(\lambda)e^{-tH(\lambda)}
$$
where $e^{-tH(\lambda)}$ is the Hermite semigroup generated by $ H(\lambda)$:
$$
e^{-tH(\lambda)}=\sum_{k=0}^{\infty}e^{-(2k+n)|\lambda|t}P_k(\lambda).
$$
In view of the results from  the preceding subsections, it follows that
$$
\int_{-\infty}^{\infty}e^{-t\mathcal{L}}f(z,w)e^{i\lambda w}\,dw=(2\pi)^{-n}|\lambda|^n\sum_{k=0}^{\infty}e^{-(2k+n)|\lambda|t}f^{\lambda}\ast_{\lambda}
\varphi_k^{\lambda}(z).
$$
If we define $q_t$ by the equation
\begin{equation}
\label{eq:expansionheat}
\int_{-\infty}^{\infty}q_t(z,w)e^{i\lambda w}\,dw=(2\pi)^{-n}|\lambda|^n\sum_{k=0}^{\infty}e^{-(2k+n)|\lambda|t}\varphi_k^{\lambda}(z)
=:q_t^{\lambda}(z)
\end{equation}
then we obtain $e^{-t\mathcal{L}}f=f\ast q_t$.
The function $q_t$ is called the heat kernel, which is known to be positive and
$$
\int_{\He^n}q_t(z,w)\,dz\,dw=1.
$$
Moreover, the series defining $q_t^{\lambda}(z)$ can be summed, giving the explicit expression
\begin{equation}
\label{eq:qtlambda}
q_t^{\lambda}(z)= (4\pi)^{-n}\Big(\frac{\lambda}{\sinh t\lambda}\Big)^ne^{-\frac14\lambda(\coth t\lambda)|z|^2},
\end{equation}
see \cite[Theorem 2.8.1]{STU}. The heat kernel $ q_t(z,w) $ satisfies the following estimate (see \cite[Proposition 2.8.2]{STU})
$$
q_t(z,w) \leq c_n t^{-n-1} e^{-\frac{a}{t}|(z,w)|^2}
$$
for some positive constants $ c_n $ and $ a$.

We are interested in proving a ground state representation for the fractional powers $ \mathcal{L}_s.$  For this we need to obtain an integral representation for $ \mathcal{L}_s $ (stated as Proposition \ref{prop:integralNONHomo}). In the Euclidean case the corresponding representation reads as
\begin{equation*}
\Delta^s f(x) = \int_0^\infty \big(f(x)-f*p_t(x)\big) t^{-s-1} dt
\end{equation*}
where $ p_t $ is the heat kernel associated to $ \Delta.$ From the explicit form of the heat kernel $ p_t $ we can easily prove the representation (see Proposition \ref{prop:pointwiseEuc})
\begin{equation*}
\Delta^s f(x) = c_{n,s} \int_0^\infty \big(f(x)-f(y)\big) |x-y|^{-n-2s} dy.
\end{equation*}
 Using the heat kernel $ q_t $ for $ \mathcal{L} $ it is not difficult to show that
\begin{equation}
\label{eq:notsuitable}
\mathcal{L}^s f(x) = \int_0^\infty  \big(f(x)- f*q_t(x)\big) t^{-s-1} dt
\end{equation}
which unfortunately cannot be simplified to yield a usable representation.

Since we need to prove such an  integral representation for $ \mathcal{L}_s $  we have to deal with certain kernels related to $ q_t$.
For $0<s<1$, let us define the modified heat kernel $\mathcal{K}_t^s(z,w)$ by the equation
\begin{equation}
\label{eq:heatmodifLsL}
\int_{-\infty}^{\infty}\mathcal{K}_t^s(z,w)e^{i\lambda w}\,dw=q_t^\lambda(z) \Big( \frac{\lambda t}{\sinh \lambda t}\Big)^{s+1},
\end{equation}
where $q_t^{\la}(z)$ is the heat kernel given in \eqref{eq:qtlambda}.
It is known that  $ \mathcal{K}_t^s $ is related to  the heat kernel associated to a generalized sublaplacian and hence it is positive.  In fact, as shown in \cite{H}, for any $ \alpha > -\frac{1}{2} $ the function $ K_{t,\alpha}(r,w) $ defined by the equation
\begin{equation*}
 \int_{-\infty}^\infty K_{t,\alpha}(r,w) e^{i\lambda w} dw = (4\pi)^{-\alpha-1} \Big(  \frac{\lambda}{\sinh \lambda t}\Big)^{\alpha+1}
e^{-\frac{1}{4} \lambda (\coth \lambda t) r^2}
\end{equation*}
is the heat kernel associated to the  generalized sublaplacian
\begin{equation*}
 -\frac{\partial^2}{\partial r^2} -\frac{2\alpha+1}{r}\frac{\partial}{\partial r}-\frac{1}{4}r^2 \frac{\partial^2}{\partial w^2},\qquad  r> 0,\,\, w \in \R
\end{equation*}
and hence  positive. Consequently,
$\mathcal{K}_t^s(z,w) = t^{s+1} K_{t,n+s+1}(|z|,w) $ is also positive. In terms of this kernel we obtain a formula for $ \mathcal{L}_s $ similar to  \eqref{eq:notsuitable} (see Proposition \ref{prop:integralNONHomo}). Moreover, the integral
$ \int_0^\infty \mathcal{K}_t^s(z,w) t^{-s-1} dt $ can be evaluated explicitly, see Proposition \ref{lem:kernelnonHomoExplicit}. Some more important (but easily proved) properties  of this kernel $\mathcal{K}_t^s(z,w)$ are stated in the following lemma.

\begin{lem}
\label{eq:heatmodifcalLs1}
Let $n\ge1$, $0<s<1$. For $(z,w)\in \He^n$, we have
\begin{equation}
\label{eq:kernelOnecal}
\int_{\He^n}\mathcal{K}_t^s(z,w)\,dz\,dw=1.
\end{equation}
Moreover, it satisfies the estimate
\begin{equation}
\label{eq:decaycal}
\mathcal{K}_t^s(z,w) \leq c_n t^{-n-1} e^{-\frac{a}{t}|(z,w)|^2},
\end{equation}
for some positive constants $ c_n $ and $ a$.
\end{lem}

\begin{proof}
By letting $ \lambda $  go to $ 0 $ in \eqref{eq:heatmodifLsL} we see that
$$
\int_{-\infty}^{\infty}\mathcal{K}_t^s(z,w)\,dw= (4\pi t)^{-n} e^{-\frac{1}{4t}|z|^2}.
$$
From this, \eqref{eq:kernelOnecal} follows immediately. For the estimate \eqref{eq:decaycal} we refer to \cite{HL} where the authors use the same argument as in \cite{STU} to prove the required estimate.
\end{proof}

In order to deal with $ \Lambda_s $ we define another  modified heat kernel $ K_t^s $ by the relation
\begin{equation}
\label{eq:heatmodifLs}
\int_{-\infty}^{\infty}K_t^s (z,w)e^{i\lambda w}\,dw = q_t^{\lambda}(z)(\coth t\lambda)\Big(\frac{\lambda t}{\sinh \lambda t}\Big)^{2-s}.
\end{equation}
We strongly believe that this kernel is positive even though we do not have a proof. However, all we need are the following properties.

\begin{lem}
\label{eq:heatmodifLsL1}
Let $n\ge1$, $0<s<1$. For $(z,w)\in \He^n$, we have
\begin{equation}
\label{eq:kernelOne}
\int_{\He^n}K_t^s(z,w)\,dz\,dw=1.
\end{equation}
Moreover, it satisfies the estimate
\begin{equation}
\label{eq:decay}
|K_t^s(z,w)| \leq c_n t^{-n-1} e^{-\frac{a}{t}|(z,w)|^2},
\end{equation}
for some positive constants $ c_n $ and $ a$.
\end{lem}
\begin{proof}
The integral \eqref{eq:kernelOne} is evaluated as above. The estimate \eqref{eq:decay} can be  proved  by modifying the proof given in \cite[Proposition 2.8.2]{STU} for the heat kernel $ q_t $ on $ \He^n$ (see \cite{HL}).
\end{proof}

We remark that the integral $ \int_0^\infty K_t^s(z,w)  t^{-s-1} dt $ can also be evaluated explicity, see Proposition~\ref{lem:kernelHomoExplicit}.

%%%%%%%%%%%%%%%%%%%%%%%%%%%%%%%%%%%%%%%%%%%%%%%%
\section{A fundamental solution for $ \mathcal{L}_s $ and the Cowling--Haagerup formula}
\label{sec:fundamental}
%%%%%%%%%%%%%%%%%%%%%%%%%%%%%%%%%%%%%%%%%%%%%%%%

Our proof of Hardy's inequality for the fractional powers of $ \mathcal{L} $ hinges on Theorem \ref{thm:CH1} below, which is essentially proved by Cowling and Haagerup in \cite[Section 3]{CH}. However, for the sake of completeness we indicate a slightly different proof here. Following \cite[p. 530]{CH} we define, for $\delta\ge0$,
\begin{equation}
\label{eq:usdelta}
u_{s,\delta}(z,w)=\Big(\big(\delta+\frac{1}{4}|z|^2\big)^2+w^2\Big)^{-\frac{s+n+1}{2}},
\end{equation}
where $(z,w)\in \He^n$. Note that
\begin{equation}
\label{eq:uszero}
u_{s,0}(z,w)=\Big(\frac{1}{16}|z|^4+w^2\Big)^{-\frac{s+n+1}{2}}= 4^{s+n+1} |(z,w)|^{-Q-2s}
\end{equation}
where $|(z,w)|=(|z|^4+ 16w^2)^{\frac14}$ is the homogeneous norm on $\He^n$ and
\begin{equation}
\label{eq:homogeneousDim}
Q=2n+2
\end{equation}
is the homogeneous dimension of $\He^n$. By an easy calculation we can check that $ u_{s,\delta} \in L^1(\He^n) $ for any $ s > 0 $ whereas $ u_{s,\delta} \in L^2(\He^n) $ for any $ s > -\frac{n+1}{2}.$

\begin{thm}
\label{thm:CH1}
Let $ \delta >0 $ and $ 0 < s < \frac{n+1}{2}.$ Then for any $ f \in W^{s,2}(\He^n) $ we have
$$  \int_{\He^n} \mathcal{L}_s f(x) u_{-s,\delta}(x) dx = (4\delta)^s  \frac{\Gamma\big(\frac{n+1+s}{2}\big)^2}{\Gamma\big(\frac{n+1-s}{2}\big)^2}\int_{\He^n} f(x) u_{s,\delta}(x) dx.$$
\end{thm}

In order to prove Theorem \ref{thm:CH1}, we need to calculate the Fourier transform of $ u_{s,\delta} $. The Fourier transform of $u_{s,1}$  was computed in \cite[Theorem 3.5]{CH} and \cite[Proposition 3.6]{CH}. Note that $ u_{s,\delta} $ is a radial function in the $ z $ variable and hence $\widehat{u}_{s,\delta}(\lambda)$ is a function of the Hermite operator $ H(\lambda)$, as explained  in Subsection \ref{sec:Hermite}. In this way, let us write
\begin{equation}
\label{eq:coefficients}
\widehat{u}_{s,\delta}(\lambda) = \sum_{k=0}^\infty  c_{k,\delta}^\lambda(s) P_k(\lambda).
\end{equation}
Therefore, the task boils down to computing the coefficients $ c_{k,\delta}^\lambda(s)$. As we have mentioned this has been done already in \cite{CH} for the case $\delta=1$ but for  the sake of completeness we include a different, more general proof here.  This result is stated in Proposition \ref{prop:Ftransform}. As a consequence, we prove Theorem \ref{thm:CH1} which, in its turn, is a key ingredient in proving Hardy's inequality for $ \mathcal{L}_s$. Moreover, at the end of this section, we also obtain a closed form expression for the fundamental solution of the operator $ \mathcal{L}_s $ which is needed in the proof of Hardy's inequality for $ \mathcal{L}_{s}^{-1}\mathcal{L}.$

The coefficients $ c_{k,\delta}^\lambda(s)$ involve an auxiliary function which is given by the following integral: for $ a, b \in \R^+ $ and $ c \in \R $ we define
\begin{equation}
\label{eq:L}
L(a,b,c) = \int_{0}^\infty e^{-a(2x+1)} x^{b-1}\big(1+x\big)^{-c} dx.
\end{equation}
The following proposition expresses  $c_{k,\delta}^\lambda(s)$  in terms of $ L.$

\begin{prop}
\label{prop:Ftransform}
For $ \delta > 0 $ and $ 0 < s < \frac{n+1}{2}.$ we have
$$  c_{k,\delta}^\lambda(s)  = \frac{(2\pi)^{n+1} |\lambda|^s}{\Gamma\big(\frac12(n+1+s)\big)^2}  L\Big(\delta |\lambda|, \frac{2k+n+1+s}{2},\frac{2k+n+1-s}{2}\Big),$$
where $c_{k,\delta}^\lambda(s)$ are the  coefficients appearing in the formula for  $\widehat{u}_{s,\delta}(\lambda) $ in \eqref{eq:coefficients}.
\end{prop}

\begin{proof}
 We begin with the following generating function identity for the Laguerre functions of type $n-1$, valid for $ |w| < 1$ (see \cite[(1.4.24)]{STH})
$$  \sum_{k=0}^\infty  w^k L_k^{n-1}\Big(\frac{1}{2}r^2\Big)e^{-\frac{1}{4}r^2} =  (1-w)^{-n} e^{-\frac{1}{4}\frac{1+w}{1-w}r^2}.$$
By taking $ w = \frac{x}{x+|\lambda|} $ and changing $ r^2 $ into $ |\lambda| r^2 $ we obtain
\begin{equation}
\label{eq:expansionG}
\sum_{k=0}^\infty  \bigg(\frac{x}{x+|\lambda|}\bigg)^k  L_k^{n-1}\Big( \frac12|\lambda|r^2\Big)e^{-\frac14|\lambda|r^2} =  |\lambda|^{-n} (x+|\lambda|)^n
e^{-\frac14(2x+|\lambda|)r^2}.
\end{equation}
For functions $ f,g $ defined on $ (0, \infty) $ let $ F, G $ be their Laplace transforms defined by
$$
F(a+ib) = \int_0^\infty e^{-(a+ib)x} f(x) dx,  \quad  G(a+ib) = \int_0^\infty e^{-(a+ib)x} g(x) dx,   \quad a > 0, \,\,b\in \R.
$$
Take $ \beta = \frac{1}{2}(n+1+s) $. Then with $ f(x) = g(x) = \Gamma(\beta)^{-1} x^{\beta-1} e^{-\delta x}, \,\,x > 0 $,  we have
\begin{equation}
\label{eq:Laplace}
F(a+ib) = G(a+ib) = (\delta+a+ib)^{-\beta}.
\end{equation}
On the other hand, it can be checked, see \cite[Lemma 3.4]{CH}, that
\begin{equation}
\label{eq:compos}
\int_{-\infty}^\infty F(a+ib)\overline{G(a+ib)}e^{-i|\lambda| b} db = 2\pi \int_0^\infty f(x)g(x+|\lambda|)e^{-a(2x+|\lambda|)} dx.
\end{equation}
 Therefore, by \eqref{eq:Laplace} and \eqref{eq:compos}, with $ a = \frac14r^2 $, we get
$$ \int_{-\infty}^\infty \Big(\big(\delta+\frac14r^2\big)^2+b^2\Big)^{-\frac12 (n+1+s)} e^{-i|\lambda| b} db = 2\pi \int_0^\infty f(x)g(x+|\lambda|)e^{-\frac14(2x+|\lambda|)r^2} dx.$$
As the function $ u_{s,\delta}(z,w) $ in \eqref{eq:usdelta} is even in the $ w $ variable, the above formula means, by taking into account  \eqref{eq:inverseFT},
\begin{equation}
\label{eq:usdeltaLambda}
(u_{s,\delta})^\lambda(z) = 2\pi \int_0^\infty f(x)g(x+|\lambda|)e^{-\frac14(2x+|\lambda|)|z|^2} dx.
\end{equation}
Using the expansion \eqref{eq:expansionG}, we can write
$$ e^{-\frac14(2x+|\lambda|)|z|^2} = |\lambda|^n (x+|\lambda|)^{-n} \sum_{k=0}^\infty \Big(\frac{x}{x+|\lambda|}\Big)^k L_k^{n-1}\Big(\frac12|\lambda||z|^2\Big)e^{-\frac14|\lambda||z|^2},
$$
so plugging this in \eqref{eq:usdeltaLambda} we obtain
$$ (u_{s,\delta})^\lambda(z) = (2\pi)^{-n} |\lambda|^n  \sum_{k=0}^\infty c_{k,\delta}^\lambda(s)  L_k^{n-1}\Big(\frac12|\lambda||z|^2\Big)e^{-\frac14|\lambda||z|^2} $$
where the coefficients are given by
\begin{align*}  c_{k,\delta}^\lambda(s)  &= (2\pi)^{n+1} \int_0^\infty f(x)g(x+|\lambda|) (x+|\lambda|)^{-n-k} x^k  dx\\
&= \frac{(2\pi)^{n+1}}{\Gamma\big(\frac12(n+1+s)\big)^2} \int_0^\infty e^{-\delta(2x+|\lambda|)} x^{\beta+k-1} (x+|\lambda|)^{\beta-k-n-1} dx.
\end{align*}
After simplification, in view of \eqref{eq:L}, we get
$$  c_{k,\delta}^\lambda(s)  =  \frac{ (2\pi)^{n+1}  |\lambda|^s}{\Gamma\big(\frac12(n+1+s)\big)^2}  L\Big(\delta |\lambda|, \frac{2k+n+1+s}{2},\frac{2k+n+1-s}{2}\Big).
$$
Moreover, by \eqref{eq:FTbis},
$$ \widehat{u}_{s,\delta}(\lambda) = \int_{\C^n} (u_{s,\delta})^\lambda(z) \pi_{\lambda}(z,0) dz,
$$
and using \eqref{eq:Laguerre} and \eqref{eq:WeylLaguerre}, i.e., the fact that
$$ \int_{\C^n} L_k^{n-1}\Big(\frac{1}{2}|\lambda| |z|^2\Big) e^{-\frac{1}{4}|\lambda||z|^2} \pi_\lambda(z,0) dz =  (2\pi)^n  |\lambda|^{-n} P_k(\lambda) $$
we immediately get, in view of the expansion for $ (u_{s,\delta})^\lambda(z), $
$$ \widehat{u}_{s,\delta}(\lambda) =    \sum_{k=0}^\infty c_{k,\delta}^\lambda(s) P_k(\lambda).$$
This completes the proof of the proposition.
\end{proof}

According to \cite[Proposition 3.6]{CH} the function $ L $ satisfies the following identity
$$ \frac{(2\lambda)^a}{\Gamma(a)}L(\lambda,a,b) = \frac{(2\lambda)^b}{\Gamma(b)}L(\lambda,b,a) $$ for all $ a, b \in \C $ and $ \lambda >0.$
Using this identity and the formula for $ c_{k,\delta}^\lambda(s) $ given in Proposition \ref{prop:Ftransform} we obtain the following relation between $ c_{k,\delta}^\lambda(s) $ and $ c_{k,\delta}^\lambda(-s).$

\begin{prop}
\label{prop:cs}
For $ \delta > 0 $ and $ 0 < s < \frac{n+1}{2} $ we have
$$  c_{k,\delta}^\lambda(-s)  =   (2\delta)^s |\lambda|^{-s} \frac{\Gamma\big(\frac{n+1+s}{2}\big)^2}{\Gamma\big(\frac{n+1-s}{2}\big)^2} \frac{\Gamma\big(\frac{2k+n}{2}+\frac{1-s}{2}\big)}{\Gamma\big(\frac{2k+n}{2}+\frac{1+s}{2}\big)}
 c_{k,\delta}^\lambda(s).
$$

\end{prop}

\begin{proof}[Proof of Theorem \ref{thm:CH1}]
In view of Plancherel theorem for the Fourier transform on $ \He^n $ given in \eqref{eq:Plancherel}, we only have to show that
$$ \widehat{(\mathcal{L}_{s} u_{-s,\delta})}(\lambda) = (4\delta)^s  \frac{\Gamma\big(\frac{n+1+s}{2}\big)^2}{\Gamma\big(\frac{n+1-s}{2}\big)^2} \widehat{u}_{s,\delta}(\lambda) $$ for any $ \lambda \in \R^*$.
Now it is easy to see that Theorem \ref{thm:CH1} follows from Proposition \ref{prop:cs}. Indeed, we have
$$ \widehat{u}_{-s,\delta}(\lambda) =    \sum_{k=0}^\infty c_{k,\delta}^\lambda(-s) P_k(\lambda).$$
By \eqref{eq:multiplier} and Proposition \ref{prop:cs} it immediately follows that
$$   \widehat{(\mathcal{L}_{s} u_{-s,\delta})}(\lambda)= (4\delta)^s  \frac{\Gamma\big(\frac{n+1+s}{2}\big)^2}{\Gamma\big(\frac{n+1-s}{2}\big)^2}  \sum_{k=0}^\infty c_{k,\delta}^\lambda(s) P_k(\lambda) =(4\delta)^s  \frac{\Gamma\big(\frac{n+1+s}{2}\big)^2}{\Gamma\big(\frac{n+1-s}{2}\big)^2} \widehat{u}_{s,\delta}(\lambda).$$
And this proves the theorem.
\end{proof}

In view of Theorem \ref{thm:CH1}, $\mathcal{L}_{s}^{-1}$ occurs as an intertwining operator between $\widehat{u}_{s,\delta}(\lambda)$ and $\widehat{u}_{-s,\delta}(\lambda)$. The family of functions $ u_{s,\delta} $ are defined even for complex values of $ s $ and they are locally integrable as long as $ \Re(s) <0.$ It has a meromorphic continuation as a distribution for other values of $ s.$  This justifies that we can apply Proposition \ref{prop:Ftransform} to $c_{k,\delta}^\lambda(-s)$. Assuming that $ 0 < s < \frac{n+1}{2}$ and letting $ \delta $ tend to $ 0 $ in the formula for $
c_{k,\delta}^\lambda(-s) $ we obtain
$$
\widehat{u_{-s,0}}(\lambda) =  \frac{(2\pi)^{n+1} |\lambda|^{-s}}{\Gamma\big(\frac12(n+1-s)\big)^2}  \sum_{k=0}^\infty L\Big(0, \frac{2k+n+1-s}{2},\frac{2k+n+1+s}{2}\Big) P_k(\lambda).
$$
It can be easily checked that
$$
L\Big(0, \frac{2k+n+1-s}{2},\frac{2k+n+1+s}{2}\Big) = \Gamma(s) \frac{\Gamma\big(\frac{2k+n}{2}+\frac{1-s}{2}\big)}{\Gamma\big(\frac{2k+n}{2}+\frac{1+s}{2}\big)}.
$$
This together with \eqref{eq:multiplier} means that
$$
\widehat{(\mathcal{L}_{s} u_{-s,0})}(\lambda) =   \frac{(2\pi)^{n+1} 2^s \Gamma(s)}{\Gamma\big(\frac12(n+1-s)\big)^2} \operatorname{Id}.
$$
In other words, by \eqref{eq:uszero}, we have that the function
$$
\frac{\Gamma\big(\frac12(n+1-s)\big)^2}{(2\pi)^{n+1} 2^s \Gamma(s)}u_{-s,0}(z,w) =  \frac{2^{n+1-3s}}{\pi^{n+1} \Gamma(s)} \Gamma\Big(\frac12(n+1-s)\Big)^2 |(z,w)|^{-Q+2s}
$$
is a fundamental solution for the operator $ \mathcal{L}_s.$ When $ s = 1$ this reduces to
$$
\frac{2^{n-2}}{\pi^{n+1}}\Gamma\Big(\frac{n}{2}\Big)^2|(z,w)|^{-Q+2}
$$
which is the fundamental solution of $ \mathcal{L} $ found by Folland in \cite{F}.  Let us denote the fundamental solution of $\mathcal{L}_s$ by $g_s$. Thus, summarizing,  for $x=(z,w)\in \He^n$, the function
\begin{equation}
\label{eq:fundSolution}
g_s(x)=\frac{2^{n+1-3s}\Gamma\big(\frac{n+1-s}{2}\big)^2}{\pi^{n+1}\Gamma(s)}|x|^{-Q+2s},
\end{equation}
is the fundamental solution of $\mathcal{L}_s$, i.e., it satisfies $\mathcal{L}_sg_s=\delta_0$, where $\delta_0$ is the Dirac delta distribution with support at $ 0$.

%%%%%%%%%%%%%%%%%%%%%%%%%%%%%%%%%%%%%%%%%%%%%%%%
\section{Integral representations}
\label{sec:integralRep}
%%%%%%%%%%%%%%%%%%%%%%%%%%%%%%%%%%%%%%%%%%%%%%%%

In order to prove Hardy's inequalities in the Heisenberg group, we will follow some ideas used by Frank et al  \cite{FLS} in the case of the Laplacian on $ \R^n $. Therefore, we need to establish ground state representations for the operators $ \Lambda_{s} $ and $ \mathcal{L}_s.$ These ground state representations will be proved in the next section as consequences of integral representations for  $ \Lambda_s $ and $ \mathcal{L}_s $ which we show in this section. Once again, we remark that the way to get the integral representations is based on the definitions with the heat semigroup.

Along the section, we will make use of several formulas and identities. We collect them here altogether.

We will use the identity (see \cite[p. 382, 3.541.1]{GR})
\begin{equation}
\label{eq:GR0}
\int_0^{\infty}e^{-\mu t}\sinh^\nu\beta t\,dt=\frac{1}{2^{\nu+1}}\frac{\Gamma\Big(\frac{\mu}{2\beta}-\frac{\nu}{2}\Big)\Gamma(\nu+1)}
{\Gamma\Big(\frac{\mu}{2\beta}+\frac{\nu}{2}+1\Big)},
\end{equation}
which is valid for $\operatorname{Re}\beta>0$, $\operatorname{Re}\nu>-1$, $\operatorname{Re}\mu>\operatorname{Re}\beta\nu$.

In order to evaluate several integrals that arise later, we shall use (see \cite[p. 498, 3.944.6]{GR})
\begin{equation}
\label{eq:GR1}
\int_0^{\infty}x^{\mu-1}e^{-\beta x}(\cos \delta x)\,dx=\frac{\Gamma(\mu)}{(\delta^2+\beta^2)^{\mu/2}}\cos\Big(\mu\arctan \frac{\delta}{\beta}\Big),
\end{equation}
valid for $\operatorname{Re}\mu>0$, $\operatorname{Re}\beta>|\operatorname{Im}\delta|$.
Also, we have the formula \cite[p. 406, 3.663.1]{GR}
\begin{equation}
\label{eq:GR2}
\int_0^u(\cos x-\cos u)^{\nu-\frac12}\cos ax\,dx=\sqrt{\frac{\pi}{2}}(\sin u)^{\nu}\Gamma\Big(\nu+\frac12\Big)P_{a-\frac12}^{-\nu}(\cos u),
\end{equation}
valid for $ \operatorname{Re}\nu>-\frac12$, $a>0$, $0<u<\pi$, where $P_{a-\frac12}^{-\nu}$ is an associated Legendre function of the first kind (see for instance \cite[Sections 8.7-8.8]{GR}).
On the other hand, we have \cite[p. 406, 3.663.2]{GR}
\begin{equation}
\label{eq:GR3}
\int_0^u(\cos x-\cos u)^{\nu-1}\cos [(\nu+\beta)]x\,dx=
\frac{\sqrt{\pi}\Gamma(\beta+1)\Gamma(\nu)\Gamma(2\nu)(\sin u)^{2\nu-1}}
{2^{\nu}\Gamma(\beta+2\nu)\Gamma \big(\nu+\frac12\big)}C_{\beta}^{\nu}(\cos u),
\end{equation}
valid for $ \operatorname{Re}\nu>0$, $\operatorname{Re}\beta>-1$, $0<u<\pi$, where $C_{\beta}^{\nu}$ is a Gegenbauer polynomial (see for instance \cite[Section 8.93]{GR}).

Recall the following representation for the associated Legendre function (\cite[p. 969, 8.755]{GR})
\begin{equation}
\label{eq:Legendre}
P_{\nu}^{-\nu}(\cos\varphi)=\frac{\big(\frac{\sin\varphi}{2}\big)^{\nu}}{\Gamma(1+\nu)}.
\end{equation}

Finally, it is known that
\begin{equation}
\label{eq:Gegenbauer}
C_{1}^{\nu}(\cos \gm)= 2\nu\cos\gm,
\end{equation}
see for instance \cite[Section 8.93]{GR}.

%%%%%%%%%%%%%%%%%%%%%%%%%%%%%%%%%%%%%%%%%%%%%%%%
\subsection{The non-homogeneous case: the operator $\mathcal{L}_s$}
%%%%%%%%%%%%%%%%%%%%%%%%%%%%%%%%%%%%%%%%%%%%%%%%
In  this subsection we prove an integral representation for the operator $ \mathcal{L}_s $.
Recall the kernel  $ \mathcal{K}_t^s $ \eqref{eq:heatmodifLsL} whose properties have been stated in Lemma 2.1. In terms of this kernel we define another kernel $ \mathcal{K}_s $ by
\begin{equation}
\label{eq:calKsIntegral}
\mathcal{K}_s(z,w) = \int_0^\infty \mathcal{K}_t^s(z,w) t^{-s-1} dt.
\end{equation}
This kernel can be explicitly calculated (see Proposition \ref{lem:kernelnonHomoExplicit}):
\begin{equation}
\label{eq:calKs}
\mathcal{K}_s(z,w)= c_{n,s} |(z,w)|^{-Q-2s}
\end{equation}
where  $ c_{n,s}$ is a positive constant which can  be explicitly determined.  Observe that the kernel $ \mathcal{K}_s $ is homogeneous of degree $ -Q-2s$.
We obtain an integral representation for the operator $\mathcal{L}_s$ in the proposition below.

\begin{prop}
\label{prop:integralNONHomo}
Let $n\ge 1$ and $0<s<1$. Then for all $f \in W^{s,2}(\He^n) $  we have
$$ \mathcal{L}_sf = \int_0^\infty (f-f\ast \mathcal{K}_t^s) t^{-s-1} dt.$$
Moreover, the following pointwise representation is valid for all $ f \in C_0^\infty(\He^n):$
\begin{equation*}
%\label{eq:Ls}
\mathcal{L}_sf(x)=\frac{1}{|\Gamma(-s)|}\int_{\He^n}\big(f(x)-f(y)\big)\mathcal{K}_s(y^{-1}x)\,dy,
\end{equation*}
where $\mathcal{K}_s(x)$ is given  in \eqref{eq:calKs}.
\end{prop}

\begin{proof}
We begin with the identity \eqref{eq:GR0}, taking $\nu=-s$, $\beta=1$ and turning $\mu\to \mu+1$. So, we have the formula
$$
2^{1-s}\int_0^{\infty}e^{-(\mu+1)t}(\sinh t)^{-s}\,dt=\frac{\Gamma(1-s)\Gamma\big(\frac{\mu}{2}+\frac{1+s}{2}\big)}
{\Gamma\big(\frac{\mu}{2}+\frac{1-s}{2}+1\big)},
$$
which gives
\begin{equation}
\label{eq:formula}
(\mu+1-s)\int_0^{\infty}e^{-(\mu+1)t}(\sinh t)^{-s}\,dt=\frac{2^{s}\Gamma(1-s)\Gamma\big(\frac{\mu}{2}+\frac{1+s}{2}\big)}
{\Gamma\big(\frac{\mu}{2}+\frac{1-s}{2}\big)}.
\end{equation}
Moreover, an integration by parts gives
\begin{align*}
(\mu+1)\int_0^{\infty}e^{-(\mu+1)t}(\sinh t)^{-s}\,dt&=\int_0^{\infty}\frac{d}{dt}(1-e^{-(\mu+1)t})(\sinh t)^{-s}\,dt\\
&=s\int_0^{\infty}(1-e^{-(\mu+1)t})(\sinh t)^{-s-1}(\cosh t)\,dt.
\end{align*}
Therefore, plugging the latter into \eqref{eq:formula}, we get
\begin{align*}
\frac{2^s\Gamma(1-s)\Gamma\big(\frac{\mu}{2}+\frac{1+s}{2}\big)}
{\Gamma\big(\frac{\mu}{2}+\frac{1-s}{2}\big)}&=s\int_0^{\infty}\big(\cosh t-e^{-(\mu+1)t}(\cosh t+\sinh t)\big)(\sinh t)^{-s-1}\,dt\\
&=s\int_0^{\infty}\big(\cosh t-e^{-\mu t}\big)(\sinh t)^{-s-1}\,dt\\
&=s\int_0^{\infty}\big(\cosh t-1\big)(\sinh t)^{-s-1}\,dt+s\int_0^{\infty}\big(1-e^{-\mu t}\big)(\sinh t)^{-s-1}\,dt\\
&=c_1 s+s\int_0^{\infty}\big(1-e^{-\mu t}\big)(\sinh t)^{-s-1}\,dt
\end{align*}
where $c_1$ is the constant given by
$$
c_1:=\int_0^{\infty}\big(\cosh t-1\big)(\sinh t)^{-s-1}\,dt.
$$
Thus, by taking $\mu=2k+n$ and changing $t$ into $|\lambda|t$, we have
\begin{align*}
\frac{2^s\Gamma(1-s)}{s}\frac{\Gamma\big(\frac{2k+n}{2}+\frac{1+s}{2}\big)}
{\Gamma\big(\frac{2k+n}{2}+\frac{1-s}{2}\big)}&=c_1 +\int_0^{\infty}\big(1-e^{-(2k+n) t}\big)(\sinh t)^{-s-1}\,dt\\
&=c_1 +|\lambda|\int_0^{\infty}\big(1-e^{-(2k+n) |\lambda|t}\big)(\sinh t|\lambda|)^{-s-1}\,dt.
\end{align*}
We now multiply both sides by $ |\lambda|^{s}f^{\lambda}\ast_{\lambda}\varphi_k^{\lambda}(z)$. Thus
\begin{multline*}
\frac{\Gamma(1-s)}{s}(2|\lambda|)^{s}\frac{\Gamma\big(\frac{2k+n}{2}+\frac{1+s}{2}\big)}
{\Gamma\big(\frac{2k+n}{2}+\frac{1-s}{2}\big)}f^{\lambda}\ast_{\lambda}\varphi_k^{\lambda}(z)=c_1|\lambda|^{s}f^{\lambda}\ast_{\lambda}\varphi_k^{\lambda}(z)\\
 +\int_0^{\infty}\big(1-e^{-(2k+n) |\lambda|t}\big)\Big(\frac{t|\lambda|}{\sinh t\lambda}\Big)^{s+1}f^{\lambda}\ast_{\lambda}\varphi_k^{\lambda}(z)t^{-s-1}\,dt.
\end{multline*}
Summing over $k$, and taking into account \eqref{eq:expansion} and \eqref{eq:expansionheat}, we obtain
\begin{multline*}
\frac{\Gamma(1-s)}{s}(2|\lambda|)^{s}(2\pi)^{-n}|\lambda|^n\sum_{k=0}^{\infty}
\frac{\Gamma\big(\frac{2k+n}{2}+\frac{1+s}{2}\big)}
{\Gamma\big(\frac{2k+n}{2}+\frac{1-s}{2}\big)}
f^{\lambda}\ast_{\lambda}\varphi_k^{\lambda}(z)\\
=c_1|\lambda|^{s}f^{\lambda}(z)
 +\int_0^{\infty}\big(f^{\lambda}(z)-f^{\lambda}\ast_{\lambda}q_t^{\lambda}(z)\big)
 \Big(\frac{t\lambda}{\sinh t\lambda}\Big)^{s+1}t^{-s-1}\,dt,
\end{multline*}
where $q_t^{\lambda}$ is as in \eqref{eq:qtlambda}.

We now rewrite the last integral as a sum of the following two integrals:
$$ A =f^{\lambda}(z)\int_0^{\infty}\Big(\Big(\frac{t\lambda}{\sinh t\lambda}\Big)^{s+1}-1\Big)t^{-s-1}\,dt,$$
$$ B = \int_0^{\infty}\Big(f^{\lambda}(z)-\Big(\frac{t\lambda}{\sinh t\lambda}\Big)^{s+1}f^{\lambda}\ast_{\lambda}\varphi_k^{\lambda}(z)\Big)t^{-s-1}\,dt.$$
Note that the first integral  $ A $ is  equal to
$$
|\lambda|^sf^{\lambda}(z)\int_0^{\infty}\Big(\Big(\frac{t}{\sinh t}\Big)^{s+1}-1\Big)t^{-s-1}\,dt=:-c_2|\lambda|^sf^{\lambda}(z).
$$
It happens that $c_1=c_2$. Indeed,
\begin{align*}
c_1-c_2&=\int_0^{\infty}\big(\cosh t-1\big)(\sinh t)^{-s-1}\,dt+\int_0^{\infty}\Big(\Big(\frac{t}{\sinh t}\Big)^{s+1}-1\Big)t^{-s-1}\,dt\\
&=\int_0^{\infty}\big((\cosh t)(\sinh t)^{-s-1}-t^{-s-1}\big)\,dt.
\end{align*}
Consider the integral
$$
\int_{\delta}^{\infty}(\cosh t)(\sinh t)^{-s-1}\,dt=\int_{\sinh \delta}^{\infty}t^{-s-1}\,dt=\int_{\delta}^{\infty}t^{-s-1}\,dt-\int_{\delta}^{\sinh \delta}t^{-s-1}\,dt.
$$
This gives
$$
\int_{\delta}^{\infty}\big((\cosh t)(\sinh t)^{-s-1}-t^{-s-1}\big)\,dt=-\int_{\delta}^{\sinh \delta}t^{-s-1}\,dt,
$$
which converges to $0$ as $\delta\to0$.

Therefore, by \eqref{eq:inverseFT} and \eqref{eq:calKs}, the second
integral $ B $ takes the form
$$
\int_0^{\infty}\int_{-\infty}^{\infty}\big(f(z,w)-f\ast \mathcal{K}_t^s(z,w)\big)e^{i\lambda w}\,dwt^{-s-1}\,dt.$$
Consequently, by the spectral definition of $\mathcal{L}_s$ in\eqref{eq:LsSpec}, since $\frac{s}{\Gamma(1-s)}=\frac{1}{|\Gamma(-s)|}$, we obtain
$$
\mathcal{L}_sf(z,w)=\frac{1}{|\Gamma(-s)|}\int_0^{\infty}\big(f(z,w)-f\ast \mathcal{K}_t^s(z,w)\big)t^{-s-1}\,dt.
$$
The integral has to be interpreted as the Bochner integral of the $ L^2(\He^n) $ valued function $ t \rightarrow f-f\ast \mathcal{K}_t^s$. By Lemma \ref{eq:heatmodifcalLs1}, we have
$$
f(x)-f\ast \mathcal{K}_t^s(x)=f(x)-\int_{\He^n}f(y)\mathcal{K}_t^s(y^{-1}x)\,dy=\int_{\He^n}\big(f(x)-f(y)\big)\mathcal{K}_t^s(y^{-1}x)\,dy.
$$
Thus we have proved  the representation
$$
\mathcal{L}_sf(x) = \frac{1}{|\Gamma(-s)|}\int_0^\infty \Big( \int_{\He^n} (f(x)-f(y))\mathcal{K}_t^s(y^{-1}x) dy\Big) t^{-s-1} dt.
$$
We can interchange the order of integration: this is justified by using the stratified mean value theorem (see \cite[(1.41)]{FS}) under the assumption that $ f \in C_0^\infty$. Then by \eqref{eq:calKsIntegral},
we obtain the required integral representation. By Proposition \ref{lem:kernelnonHomoExplicit}, the kernel $\mathcal{K}_s$ is given by \eqref{eq:calKs}. The proof is complete.
\end{proof}

In the next proposition we prove the explicit form of the kernel $\mathcal{K}_s$. We are inspired by the ideas in \cite{AM}.

\begin{prop}
\label{lem:kernelnonHomoExplicit}
Let $n\ge1$ and $0<s<1$. For $(z,w)\in \He^n$, we have
$$
\mathcal{K}_s(z,w)=c_{n,s} |(z,w)|^{-Q-2s}
$$
where the constant $ c_{n,s} $ is given by
\begin{equation}
\label{eq:cnsNONh}
c_{n,s} = 2^{n-1+3s} \pi^{-n-1} \Gamma\Big(\frac{n+s+1}{2}\Big)^2,
\end{equation}
 and $Q$ is the homogeneous dimension of $ \He^n$, given in \eqref{eq:homogeneousDim}.
\end{prop}

\begin{proof}
We start with the expression
$$
\int_{-\infty}^{\infty}\mathcal{K}_s(z,w)e^{i\lambda w}\,dw=\int_0^{\infty}q_t^{\lambda}(z)\Big(\frac{t|\lambda|}{\sinh t|\lambda|}\Big)^{s+1}t^{-s-1}\,dt.
$$
that follows from \eqref{eq:heatmodifLsL} and \eqref{eq:calKs}. By \eqref{eq:qtlambda}, and since the functions involved are even in $\lambda$, we write
$$
\int_{-\infty}^{\infty}\mathcal{K}_s(z,w)e^{i\lambda w}\,dw=(4\pi)^{-n}\int_0^{\infty}\Big(\frac{\lambda}{\sinh t\lambda}\Big)^{n+s+1}e^{-\frac14\lambda(\coth t\lambda)|z|^2}\,dt.
$$
As the Fourier transform of $\mathcal{K}_s$ in the central variable $w$ is an even function of $\lambda$ we have, after taking the Fourier transform in the variable $\lambda$,
 $$
 \mathcal{K}_s(z,w)=4^{-n}\pi^{-n-1}\int_0^{\infty}\int_0^{\infty}(\cos \lambda u)\Big(\frac{\lambda}{\sinh t\lambda}\Big)^{n+s+1}e^{-\frac14\lambda(\coth t\lambda)|z|^2}\,d\lambda\,dt.
 $$
By the change of variables $\la\to \la|z|^{-2}$, $t\to t|z|^2$, we obtain
\begin{equation}
\label{eq:homogeneity}
\mathcal{K}_s(z,|z|^2w)=|z|^{-2(n+s+1)}\mathcal{K}_s(1,w).
\end{equation}
Thus
\begin{align*}
\mathcal{K}_s(1,w)&=\int_0^{\infty}\int_0^{\infty}(\cos \la w)\Big(\frac{\la}{\sinh t \la}\Big)^{n+s+1}e^{-\frac{\la}{4}(\coth t\la)}\,dt\,d\la\\
&=4^{-n}\pi^{-n-1}\int_0^{\infty} \Big( \int_0^{\infty}(\cos \la w) \lambda^{n+s}e^{-\frac{\la}{4}(\coth t)} \,d\lambda \Big) (\sinh t )^{-n+s+1} \,dt.
\end{align*}
The integral in $\la$ can be evaluated by using \eqref{eq:GR1} with $\mu=n+s+1$, $\beta=\frac{1}{4}(\coth t)$ and $\delta=w$. Then, we get
$$
\int_0^{\infty}(\cos \la w)\la^{n+2}e^{-\frac{\la}{4}(\coth t)}\,d\la=\frac{\Gamma(n+s+1) \cos \Big((n+s+1)\arctan\Big(\frac{4w}{\coth t}\Big)\Big)}{\Big(w^2+\frac{1}{16}\coth^2t\Big)^{\frac{n+s+1}{2}}}.
$$
Thus
\begin{equation}
\label{eq:1Hom}
 \mathcal{K}_s(1,w)= \frac{\Gamma(n+s+1)}{4^n\pi^{n+1}}\int_0^{\infty}\frac{\cos \Big((n+s+1)\arctan\Big(\frac{4w}{\coth t}\Big)\Big) }{\Big(w^2+\frac{1}{16}\coth^2t\Big)^{\frac{n+s+1}{2}}}(\sinh t)^{-(n+s+1)}\,dt.
\end{equation}
With the change of variables $u=\frac{4w}{\coth t}$ we have that the latter integral equals
\begin{align*}
\int_0^{4w}&\Big(\frac{u^2}{16w^2-u^2}\Big)^{-\frac{(n+s+1)}{2}}
\Big(w^2+\frac{16w^2}{16u^2}\Big)^{-\frac{(n+s+1)}{2}}
\cos[(n+s+1)\arctan u]\frac{4w}{16w^2-u^2}\,du\\
&=4w^{-(n+s)}
\int_0^{4w}(16w^2-u^2)^{\frac{n+s-1}{2}}(1+u^2)^{-\frac{n+s+1}{2}}\cos[(n+s+1)\arctan u]\,du\\
&=4^{n+s}w^{-1}
\int_0^{4w}\Big(1-\frac{u^2}{16w^2}\Big)^{\frac{n+s-1}{2}}(1+u^2)^{-\frac{n+s+1}{2}}\cos[(n+s+1)\arctan u]\,du.
\end{align*}
Thus, with this and \eqref{eq:1Hom} we have
\begin{equation}
\label{eq:2Hom}
\mathcal{K}_s(1,w)= \frac{2^{2s} \Gamma(n+s+1)}{\pi^{n+1}}  w^{-1} I,
\end{equation}
where
$$
I:=\int_0^{4w}\Big(1-\frac{u^2}{16w^2}\Big)^{\frac{n+s-1}{2}}(1+u^2)^{-\frac{n+s+1}{2}}\cos[(n+s+1)\arctan u]\,du.
$$
Now we will see that the above integral can be explicitly computed in terms of Legendre functions.

Making a second change of variable $\arctan u=z$, the integral $I$ becomes
$$
I=\int_0^{\arctan4w}\Big(\cos^2z-\frac{\sin^2z}{16w^2}\Big)
^{\frac{n+s-1}{2}}\cos[(n+s+1)z]\,dz.
$$
We can rewrite the above integral as
\begin{align*}
I&=\int_0^{\arctan4w}\Big(\frac{1+\cos2z}{2}-\frac{1-\cos2z}{2\cdot16w^2}\Big)^{\frac{n+s-1}{2}}
\cos[(n+s+1)z]\,dz\\
&=2^{-\frac{n+s-1}{2}}\int_0^{\arctan4w}\bigg((\cos 2z)\Big(1+\frac{1}{16w^2}\Big)-\Big(\frac{1}{16w^2}-1\Big)\bigg)^{\frac{n+s-1}{2}}
\cos[(n+s+1)z]\,dz\\
&=\Big(\frac{1+16w^2}{32w^2}\Big)^{\frac{n+s-1}{2}}\int_0^{\arctan4w}\Big(\cos2z-\frac{1-16w^2}{1+16w^2}\Big)^{\frac{n+s-1}{2}}\cos[(n+s+1)z]\,dz\\
&=\frac12\Big(\frac{1+16w^2}{32w^2}\Big)^{\frac{n+s-1}{2}}\int_0^{2\arctan4w}(\cos\beta-\cos\gamma)^{\frac{n+s-1}{2}}\cos\Big[\frac{(n+s+1)}{2}\beta\Big]\,d\beta,
\end{align*}
where $\cos\gamma=\frac{1-16w^2}{1+16w^2}$. The integral can be evaluated using \eqref{eq:GR2} by taking $\nu=\frac{n+s}{2}$ and $a=\frac{n+s+1}{2}$. With this, and by the representation for the associated Legendre function \eqref{eq:Legendre}, the latter integral becomes
\begin{align*}
\sqrt{\frac{\pi}{2}}(\sin \gm)^{\frac{n+s}{2}}\Gamma\Big(\frac{n+s+1}{2}\Big)P_{\frac{n+s}{2}}^{-\frac{n+s}{2}}(\cos \gm)&=\sqrt{\frac{\pi}{2}}\Gamma\Big(\frac{n+s+1}{2}\Big)(\sin \gm)^{\frac{n+s}{2}}
\frac{(\sin\gm)^{\frac{n+s}{2}}}{2^{\frac{n+s}{2}}\Gamma\big(\frac{n+s+2}{2}\big)}\\
&=\sqrt{\frac{\pi}{2}}\frac{\Gamma\big(\frac{n+s+1}{2}\big)}
{2^{\frac{n+s}{2}}\Gamma\big(\frac{n+s+2}{2}\big)}(\sin^2\gm)^{\frac{n+s}{2}}\\
&=\sqrt{\frac{\pi}{2}}\frac{\Gamma\big(\frac{n+s+1}{2}\big)}{2^{\frac{n+s}{2}
\Gamma\big(\frac{n+s+2}{2}\big)}}\Big(\frac{8w}{1+16w^2}\Big)^{n+s},
\end{align*}
because $\sin^2\gamma=\frac{64w^2}{(1+16w^2)^2}$. This gives
\begin{equation}
\label{eq:3Hom}
 I=\frac12  \sqrt{\frac{\pi}{2}}\frac{\Gamma\big(\frac{n+s+1}{2}\big)}{2^{\frac{n+s}{2}
\Gamma\big(\frac{n+s+2}{2}\big)}}\Big(\frac{1+16w^2}{32w^2}\Big)^{\frac{n+s-1}{2}}\Big(\frac{8u}{1+16w^2}\Big)^{n+s}= \frac{\sqrt{\pi}}{2}\frac{\Gamma\big(\frac{n+s+1}{2}\big)}{
\Gamma\big(\frac{n+s+2}{2}\big)}w(1+16w^2)^{-\frac{n+s+1}{2}}.
\end{equation}
Finally, plugging \eqref{eq:3Hom} into  \eqref{eq:2Hom}, we have
$$
\mathcal{K}_s(1,w)= \frac{2^{2s} \Gamma(n+s+1)}{\pi^{n+1}}
\frac{\sqrt{\pi}}{2}\frac{\Gamma\big(\frac{n+s+1}{2}\big)}{
\Gamma\big(\frac{n+s+2}{2}\big)}(1+16w^2)^{-\frac{n+s+1}{2}},
$$
or, by \eqref{eq:homogeneity}
$$
\mathcal{K}_s(z,w)=|z|^{-2(n+1+s)}\mathcal{K}_s\Big(1,\frac{w}{|z|^2}\Big)=c_{n,s}
|(z,w)|^{-Q-2s}
$$
where the constant $ c_{n,s} $ is given by
$$
c_{n,s} = \sqrt{\pi} \frac{2^{-1+2s}\Gamma(n+s+1)}{\pi^{n+1}} \frac{\Gamma\big(\frac{n+s+1}{2}\big)}{\Gamma\big(\frac{n+s+2}{2}\big)}.
$$
By using Legendre's duplication formula
\begin{equation}
\label{eq:LegendreDup}
\sqrt{\pi} \Gamma(2z) = 2^{2z-1}\Gamma(z)\Gamma\Big(z+\frac12\Big)
\end{equation}
with $z=\frac{n+s+1}{2}$, and after simplification, we get
$$
c_{n,s} = 2^{n-1+3s} \pi^{-n-1} \Gamma\Big(\frac{n+s+1}{2}\Big)^2.
$$
The proof is complete.

\end{proof}

%%%%%%%%%%%%%%%%%%%%%%%%%%%%%%%%%%%%%%%%%%%%%%%%
\subsection{The homogeneous case: the operator $\Lambda_{s}$}
%%%%%%%%%%%%%%%%%%%%%%%%%%%%%%%%%%%%%%%%%%%%%%%%

Our goal in this subsection is to prove an integral representation for the operator $ \Lambda_s$ defined in \eqref{eq:Lambda} similar to what we have done for $ \mathcal{L}_{s} $ in Proposition \ref{prop:integralNONHomo}. It is convenient to work with $ \Lambda_{1-s} = \mathcal{L}_s^{-1}\mathcal{L} $ and so we state our results for  this operator.

Recall  the modified heat kernel $K_t^s(z,w)$ defined by \eqref{eq:heatmodifLs}.
The properties of this kernel have been stated in Lemma \ref{eq:heatmodifLsL1}.  In terms of this kernel, we define another kernel $K_s$ by
\begin{equation}
 \label{eq:KsIntegral}
 K_s(z,w)=\int_0^{\infty}K_t^s(z,w)t^{s-2}\,dt.
 \end{equation}
The latter can be explicitly computed (see Proposition \ref{lem:kernelHomoExplicit}), and it turns out to be
\begin{equation}
\label{eq:Ks}
K_s(z,w)=c_{n,s}\frac{|z|^2}{|(z,w)|^2} |(z,w)|^{-Q-2(1-s)}.
\end{equation}
where  $ c_{n,s}$ is an explicit positive constant. Observe that the kernel $ K_s $ is homogeneous of degree $ -Q-2(1-s).$

We can now prove the following integral representation for $ \Lambda_{1-s}$.

\begin{prop}
\label{prop:integralHomo}
Let $n\ge 1$ and $0<s<1$. Then for all $f \in W^{1-s,2}(\He^n) $  we have
$$ \Lambda_{1-s}f = \int_0^\infty (f-f\ast {K}_t^s) t^{-s-1} dt.$$
Moreover, the following pointwise representation is valid for all $ f \in C_0^\infty(\He^n):$
\begin{equation*}
%\label{eq:Ls}
\Lambda_{1-s}f(x)=\frac{1}{|\Gamma(s-1)|}\int_{\He^n}\big(f(x)-f(y)\big) K_s(y^{-1}x)\,dy,
\end{equation*}
where $K_s(x)$ is given  in \eqref{eq:Ks}.
\end{prop}

\begin{proof}
By taking $\nu=s-1$ and $\beta=1$ in \eqref{eq:GR0}, we have
$$
2^s\int_0^{\infty}e^{-\mu t}(\sinh t)^{s-1}\,dt=\frac{\Gamma(s)\Gamma\big(\frac{\mu}{2}+\frac{1-s}{2}\big)}
{\Gamma\big(\frac{\mu}{2}+\frac{1+s}{2}\big)}.
$$
We rewrite the above identity as
$$
2^s\int_0^{\infty}\frac{d}{dt}(1-e^{-\mu t})(\sinh t)^{s-1}\,dt=\mu\frac{\Gamma(s)\Gamma\big(\frac{\mu}{2}+\frac{1-s}{2}\big)}
{\Gamma\big(\frac{\mu}{2}+\frac{1+s}{2}\big)}.
$$
Integrating by parts we obtain
$$
\mu\frac{\Gamma(s)\Gamma\Big(\frac{\mu}{2}+\frac{1-s}{2}\Big)}
{\Gamma\Big(\frac{\mu}{2}+\frac{1+s}{2}\Big)}=2^s(1-s)\int_0^{\infty}(1-e^{-\mu t})(\cosh t)(\sinh t)^{s-2}\,dt.
$$

By an argument analogous to the one used in the proof of Proposition \ref{prop:integralNONHomo}, it can be checked that
$$
\int_{0}^{\infty}\big((\cosh t)(\sinh t)^{s-2}-t^{s-2}\big)\,dt=0.
$$
In view of this we have
$$
2^{-s}\mu\frac{\Gamma(s)\Gamma\big(\frac{\mu}{2}+\frac{1-s}{2}\big)}
{\Gamma\big(\frac{\mu}{2}+\frac{1+s}{2}\big)}=(1-s)\int_0^{\infty}\Big(1-e^{-\mu t}(\cosh t)\Big(\frac{t}{\sinh t}\Big)^{2-s}\Big)t^{s-2}\,dt.
$$
Thus we get, by taking $\mu=2k+n$ and changing $ t $ into $ |\lambda| t$ we get
\begin{align*}
2^{-s}(2k+n)&\frac{\Gamma(s)\Gamma\big(\frac{2k+n}{2}+\frac{1-s}{2}\big)}
{\Gamma\big(\frac{2k+n}{2}+\frac{1+s}{2}\big)}=(1-s) \int_0^{\infty}\Big(1-e^{-(2k+n) t}(\cosh t)\Big(\frac{t}{\sinh t}\Big)^{2-s}\Big)t^{s-2}\,dt\\
&=(1-s)|\lambda|^{s-1}\int_0^{\infty}\Big(1-e^{-(2k+n) t|\lambda|}(\cosh t|\lambda|)\Big(\frac{t|\lambda|}{\sinh t|\lambda|}\Big)^{2-s}\Big)t^{s-2}\,dt.
\end{align*}
Multiplying both sides by $(1-s)^{-1}(2\pi)^{-n}|\lambda|^{n+1-s} f^{\lambda}\ast_{\lambda}\varphi_k^{\lambda}(z)$ and summing over $k$, we see that
\begin{multline*} (1-s)^{-1}(2\pi)^{-n}|\lambda|^n\sum_{k=0}^{\infty} \frac{\Gamma(s)\Gamma\big(\frac{2k+n}{2}+\frac{1-s}{2}\big)}
{\Gamma\big(\frac{2k+n}{2}+\frac{1+s}{2}\big)}(2|\lambda|)^{-s}\big((2k+n)|\lambda|\big) f^{\lambda}\ast_{\lambda}\varphi_k^{\lambda}(z)\\
=(2\pi)^{-n}|\lambda|^n \int_0^{\infty}\Bigg(\sum_{k=0}^{\infty}f^{\lambda}\ast_{\lambda}\varphi_k^{\lambda}(z)  -\sum_{k=0}^{\infty}e^{-(2k+n)t|\lambda|}(\coth t|\lambda|)
\Big(\frac{t|\lambda|}{\sinh t|\lambda|}\Big)^{2-s} f^{\lambda}\ast_{\lambda}\varphi_k^{\lambda}(z)\Bigg)t^{s-2}\,dt.
\end{multline*}
By taking into account \eqref{eq:Linverse} and \eqref{eq:Lsinverse} on the left hand side, and \eqref{eq:expansionheat}
 on the right hand side, we obtain
$$
\int_{-\infty}^{\infty}e^{i\lambda w}\Lambda_{1-s}f(z,w)\,dw=\frac{(1-s)}{\Gamma(s)}
\int_0^{\infty}\bigg(f^{\lambda}(z)-f^{\lambda}\ast_{\lambda}q_t^{\lambda}(z)(\coth t|\lambda|)\Big(\frac{t|\lambda|}{\sinh t|\lambda|}\Big)^{2-s}\bigg)t^{s-2}\,dt.
$$
Then by \eqref{eq:heatmodifLs}, we can rewrite the above equation as
$$
\int_{-\infty}^{\infty}e^{i\lambda w}\Lambda_{1-s}f(z,w)\,dw=\frac{(1-s)}{\Gamma(s)}
\int_0^{\infty}\bigg(f^{\lambda}(z)-f^{\lambda}\ast_{\lambda}(K_t^s)^{\lambda}(z)\bigg)t^{s-2}\,dt.
$$
This simply means that
$$ \Lambda_{1-s}f(x) = \frac{(1-s)}{\Gamma(s)}
\int_0^{\infty}(f(x)-f\ast K_t^s(x))t^{s-2}\,dt.
$$
By Lemma \ref{eq:heatmodifLsL1}, we have
$$
f(x)-f\ast K_t^s(x)=f(x)-\int_{\He^n}f(y)K_t^s(y^{-1}x)\,dy=\int_{\He^n}\big(f(x)-f(y)\big)K_t^s(y^{-1}x)\,dy.
$$
Thus we have proved  the representation
$$
\Lambda_{1-s}f(x)=\frac{(1-s)}{\Gamma(s)}
\int_0^{\infty} \Big( \int_{\He^n}\big(f(x)-f(y)\big)K_t^s(y^{-1}x)\,dy\Big) t^{s-2}\,dt,
$$
or, equivalently, assuming that we could interchange the order of integration, which  can be  justified by using mean value theorem under the assumption that $ f \in C_0^\infty(\He^n)$, we get
$$
\Lambda_{1-s}f(x)=\frac{1}{|\Gamma(s-1)|}
\int_{\He^n}\big(f(x)-f(y)\big)K_s(y^{-1}x)\,dy,
$$
where $K_s$ is the kernel defined in \eqref{eq:KsIntegral}.
 By Proposition \ref{lem:kernelHomoExplicit}, the kernel $ K_s $ is given by \eqref{eq:Ks}. This completes the proof of the proposition.
\end{proof}

In the next proposition we explicitly calculate the kernel $ K_s $ and show that has the explicit form \eqref{eq:Ks}. The proof follows the lines of Proposition \ref{lem:kernelnonHomoExplicit} but with certain modifications. Since tracking the constants is important, we show a complete proof, just skipping some computations analogous to the previous ones. Observe that, from the very definition, it is not difficult to check the homogeneity of $ K_s$.

\begin{prop}
\label{lem:kernelHomoExplicit}
Let $n\ge1$ and $0<s<1$. Let $K_s(z,w)$ be given by \eqref{eq:KsIntegral}. For $(z,w)\in \He^n$, we have
$$
K_s(z,w)=c_{n,s}\frac{|z|^2}{|(z,w)|^2} |(z,w)|^{-Q-2(1-s)},
$$
where the constant $ c_{n,s} $ is given by
\begin{equation}
\label{eq:cnsHomog}
c_{n,s}= 2^{n+5-3s} \pi^{-n-1} \Gamma\Big(\frac{n-s+3}{2}\Big) \Gamma\Big(\frac{n+1-s}{2}\Big),
\end{equation}
and $Q$ is the homogeneous dimension of $\He^n$ given in \eqref{eq:homogeneousDim}.
\end{prop}

\begin{proof}
We start with the expression which defines the kernel $ K_s$, namely
$$
\int_{-\infty}^{\infty}K_s(z,w)e^{i\lambda w}\,dw=\int_0^{\infty}q_t^{\lambda}(z)(\cosh t|\lambda|)\Big(\frac{t|\lambda|}{\sinh t|\lambda|}\Big)^{2-s}t^{s-2}\,dt.
$$
that follows from \eqref{eq:heatmodifLs} and \eqref{eq:KsIntegral}. By \eqref{eq:qtlambda}, and since the functions involved are even in $\lambda$, we write
$$
\int_{-\infty}^{\infty}K_s(z,w)e^{i\lambda w}\,dw=(4\pi)^{-n}\int_0^{\infty}(\cosh t\lambda)\Big(\frac{\lambda}{\sinh t\lambda}\Big)^{n+2-s}e^{-\frac14\lambda(\coth t\lambda)|z|^2}\,dt.
$$
As the Fourier transform of $K_s$ in the central variable $w$ is an even function of $\lambda$ we have, after taking the (Euclidean) Fourier transform in the variable $\lambda$  \eqref{eq:lambdaFT},
$$
 K_s(z,w)= 4^{-n}\pi^{-n-1} \int_0^{\infty}\int_0^{\infty}(\cos \lambda w)\Big(\frac{\lambda}{\sinh t\lambda}\Big)^{n+2-s}(\cosh t\lambda)e^{-\frac14\lambda(\coth t\lambda)|z|^2}\,d\lambda\,dt.
 $$
By the change of variables $\la\to \la|z|^{-2}$, $t\to t|z|^2$, we obtain
\begin{equation}
\label{homogeneityHom}
K_s(z,|z|^2w)=|z|^{-2(n+2-s)}K_s(1,w).
\end{equation}
Thus
\begin{align*}
K_s(1,w)&=\int_0^{\infty}\int_0^{\infty}(\cos \la w)(\cosh t\lambda)\Big(\frac{\la}{\sinh t \la}\Big)^{n+2-s}e^{-\frac{\la}{4}(\coth t\la)}\,dt\,d\la\\
&=4^{-n}\pi^{-(n+1)}\int_0^{\infty} \Big( \int_0^{\infty}(\cos \la w) \lambda^{n+1-s}e^{-\frac{\la}{4}(\coth t)} \,d\lambda \Big) (\cosh t)(\sinh t )^{-n-2+s} \,dt.
\end{align*}
 The integral in $\la$ can be evaluated using \eqref{eq:GR1}, by taking $\mu=n+2-s$, $\beta=\frac{1}{4}(\coth t)$ and $\delta=w$. Then we get
$$
\int_0^{\infty}(\cos \la w)\la^{n+2-s-1}e^{-\frac{\la}{4}(\coth t)}\,d\la=\frac{\Gamma(n+2-s) \cos \Big((n-s+2)\arctan\Big(\frac{4w}{\coth t}\Big)\Big)}{\Big(w^2+\frac{1}{16}\coth^2t\Big)^{\frac{n-s+2}{2}}}.
$$
Therefore
\begin{equation}
\label{eq:1}
 K_s(1,w)= \frac{\Gamma(n+2-s)}{4^n\pi^{n+1}}\int_0^{\infty}\frac{\cos \Big((n-s+2)\arctan\Big(\frac{4w}{\coth t}\Big)\Big) }{\Big(w^2+\frac{1}{16}\coth^2t\Big)^{\frac{n-s+2}{2}}} (\cosh t)(\sinh t)^{-n-2+s}\,dt.
\end{equation}
With the change of variables $u=\frac{4w}{\coth t}$ we have that the latter integral equals
$$
4^{n-s+1}w^{-1}
\int_0^{4w}\Big(1-\frac{u^2}{16w^2}\Big)^{\frac{n-1-s}{2}}(1+u^2)^{\frac{s-n-2}{2}}\cos[(n+2-s)\arctan u]\,du.
$$
Thus, with this and \eqref{eq:1}, we have
\begin{equation}
\label{eq:2}
K_s(1,w)= \frac{2^{2-2s} \Gamma(n+2-s)}{\pi^{n+1}} w^{-1} I,
\end{equation}
where
$$
I:=\int_0^{4w}\Big(1-\frac{u^2}{16w^2})^{\frac{n-1-s}{2}}(1+u^2)^{\frac{s-n-2}{2}}\cos[(n+2-s)\arctan u]\,du.
$$
Fortunately for us, we will see that the above integral can be explicitly computed in terms of Legendre functions and Gegenbauer polynomials.

Making a second change of variable $\arctan u=z$, the integral $I$ becomes
$$
I=\int_0^{\arctan4w}(\cos^2z)^{\frac12}\Big(\cos^2z-\frac{\sin^2z}{16w^2}\Big)
^{\frac{n-s-1}{2}}\cos[(n+2-s)z]\,dz.
$$
We can rewrite the above integral as
$$
\frac12\Big(\frac{1+16w^2}{32w^2}\Big)^{\frac{n-s-1}{2}}\int_0^{2\arctan4w}(\cos \beta/2)(\cos\beta-\cos\gamma)^{\frac{n-s-1}{2}}\cos\Big[\frac{(n+2-s)}{2}\beta\Big]\,d\beta,
$$
where $\cos\gamma=\frac{1-16w^2}{1+16w^2}$. By using the formulas $\cos(a\pm b)=\cos a\cos b \mp \sin a\sin b$ with $a=\beta/2$ and $b=\frac{n-s+2}{2}$, the latter integral is given as a sum of the following two integrals:
$$
J_1:=\int_0^{2\arctan4w}(\cos\beta-\cos\gamma)^{\frac{n-s-1}{2}}
\cos\Big[\frac{(n+1-s)}{2}\beta\Big]\,d\beta
$$
and
$$
J_2:=\int_0^{2\arctan4w}(\cos\beta-\cos\gamma)^{\frac{n-s-1}{2}}
\cos\Big[\frac{(n+3-s)}{2}\beta\Big]\,d\beta.
$$
The integral $J_1$ can be evaluated using \eqref{eq:GR2}, by taking $\nu=\frac{n-s}{2}$ and $a=\frac{n+1-s}{2}$, and then representation for the associated Legendre function in \eqref{eq:Legendre}. Thus we have
\begin{align*}
J_1&=\sqrt{\frac{\pi}{2}}(\sin \gm)^{\frac{n-s}{2}}\Gamma\Big(\frac{n+1-s}{2}\Big)P_{\frac{n-s}{2}}^{-\frac{n-s}{2}}(\cos \gm)=\sqrt{\frac{\pi}{2}}\Gamma\Big(\frac{n+1-s}{2}\Big)(\sin \gm)^{\frac{n-s}{2}}
\frac{(\sin\gm)^{\frac{n-s}{2}}}{2^{\frac{n-s}{2}}\Gamma\big(\frac{n-s+2}{2}\big)}\\
&=\sqrt{\frac{\pi}{2}}\frac{\Gamma\big(\frac{n+1-s}{2}\big)}
{2^{\frac{n-s}{2}}\Gamma\big(\frac{n-s+2}{2}\big)}(\sin^2\gm)^{\frac{n-s}{2}}.
\end{align*}
On the other hand, $J_2$ can be evaluated by using \eqref{eq:GR3} with $\nu=\frac{n-s+1}{2}$ and $\beta=1$. Then we get
$$
J_2=\frac{\sqrt{\pi}\Gamma(2)\Gamma\big(\frac{n-s+1}{2}\big)\Gamma(n-s+1)
(\sin\gm)^{n-s}}{2^{\frac{n-s+1}{2}}\Gamma(n-s+2)\Gamma \big(\frac{n-s+2}{2}\big)}C_{1}^{\frac{n-s+1}{2}}(\cos \gm).
$$
With the identity \eqref{eq:Gegenbauer}, and after simplying, we arrive at
$$
J_2=\sqrt{\frac{\pi}{2}}\frac{\Gamma\big(\frac{n+1-s}{2}\big)}{2^{\frac{n-s}{2}
\Gamma\big(\frac{n-s+2}{2}\big)}}(\sin^2\gm)^{\frac{n-s}{2}}(\cos \gm).
$$
Thus
$$
J_1+J_2=\sqrt{\frac{\pi}{2}}\frac{\Gamma\big(\frac{n+1-s}{2}\big)}{2^{\frac{n-s}{2}
\Gamma\big(\frac{n-s+2}{2}\big)}}(\sin^2\gm)^{\frac{n-s}{2}}(1+\cos \gm).
$$
Since $\cos\gamma=\frac{1-16w^2}{1+16w^2}$, we have $\sin^2\gamma=\frac{64w^2}{(1+16w^2)^2}$ and so
$$
J_1+J_2=\sqrt{\frac{\pi}{2}}\frac{\Gamma\big(\frac{n+1-s}{2}\big)}{2^{\frac{n-s}{2}
\Gamma\big(\frac{n-s+2}{2}\big)}}\Big(\frac{8w}{1+16w^2}\Big)^{n-s}\Big(\frac{2}{1+16w^2}\Big).
$$
This gives
\begin{equation}
\label{eq:3}
I=4\sqrt{\pi}\frac{\Gamma\big(\frac{n+1-s}{2}\big)}{
\Gamma\big(\frac{n-s+2}{2}\big)}w(1+16w^2)^{-\frac{n-s+3}{2}}.
\end{equation}
Finally, plugging \eqref{eq:3} into \eqref{eq:2}, we have
$$
K_s(1,w)=  4\sqrt{\pi} \frac{2^{2-2s}\Gamma(n+2-s)}{\pi^{n+1}}
\frac{\Gamma\big(\frac{n+1-s}{2}\big)}{
\Gamma\big(\frac{n-s+2}{2}\big)}(1+16w^2)^{-\frac{n-s+3}{2}},
$$
or, by \eqref{homogeneityHom}
\begin{align*}
K_s(z,w)=|z|^{-2(n+2-s)}K_s\Big(1,\frac{w}{|z|^2}\Big)=c_{n,s}\frac{|z|^2}{|(z,w)|^2}
|(z,w)|^{-Q-2(1-s)}
\end{align*}
where the constant $ c_{n,s} $ is given by
$$
c_{n,s} = 4\sqrt{\pi} \frac{2^{2-2s}\Gamma(n+2-s)}{\pi^{n+1}} \frac{\Gamma\big(\frac{n+1-s}{2}\big)}{\Gamma\big(\frac{n-s+2}{2}\big)}.
$$
By using Legendre's duplication formula \eqref{eq:LegendreDup}
 and simplifying we get
$$ c_{n,s} = 2^{n+5-3s} \pi^{-n-1} \Gamma\Big(\frac{n-s+3}{2}\Big) \Gamma\Big(\frac{n+1-s}{2}\Big).$$
\end{proof}

%%%%%%%%%%%%%%%%%%%%%%%%%%%%%%%%%%%%%%%%%%%%%%%%
\section{Ground state representations and Hardy inequalities }
\label{sec:groundHardy}
%%%%%%%%%%%%%%%%%%%%%%%%%%%%%%%%%%%%%%%%%%%%%%%%

This section contains the proofs of our main results, namely,  the Hardy inequalities for both operators $\Lambda_{s}$ and $\mathcal{L}_s$. Our proofs are fashioned after the one presented in \cite{FLS} for the case of the Euclidean Laplacian. From the integral representations obtained for $ \mathcal{L}_s $ and $ \Lambda_s$ in the previous section we first prove the so called ground state representations for these operators. With these and Theorem \ref{thm:CH1}, Hardy inequalities then become immediate corollaries. We first present a simple proof of Theorem \ref{thm:HardynonH} following a suggestion given by the referee. This proof is short and elegant. Then we give another proof via Theorem \ref{thm:gsr} based on  ideas from \cite{FLS}, which gives some improvements:  it requires an integral representation for the operator $\mathcal{L}_s$ (which is provided in Proposition \ref{prop:integralNONHomo} and it is of independent interest, although it gives a restriction on the parameter $s$), and moreover it delivers an explicit expression for the error in the Hardy inequality.

%%%%%%%%%%%%%%%%%%%%%%%%%%%%%%%%%%%%%%%%%%%%%%%%
\subsection{Proof of Theorems \ref{thm:HardynonH} and \ref{thm:gsr}}
%%%%%%%%%%%%%%%%%%%%%%%%%%%%%%%%%%%%%%%%%%%%%%%%

The result of Theorem \ref{thm:CH1} can be restated as $ \mathcal{L}_s u_{-s,\delta} =  C_{s,\delta} u_{s,\delta} $, valid for $ -\frac{n+1}{2} < s < \frac{n+1}{2} $, where $ C_{s,\delta} = (4\delta)^s \frac{\Gamma(\frac{Q+2s}{4})^2}{\Gamma(\frac{Q-2s}{4})^2}$ and $ Q $ is the homogeneous dimension of $ \He^n$ given in \eqref{eq:homogeneousDim}. In particular, we have
$ \mathcal{L}_{-s} u_{s,\delta} =  C_{-s,\delta} u_{-s,\delta} $ for $ 0 < s < \frac{n+1}{2}.$ Let $ v(z,w) = (\delta+\frac{1}{4}|z|^2)^2+w^2 $ so that $ u_{s,\delta} = v^{-s/2}u_{0,\delta}.$  It then follows that the integral operator $ T_sf = v^{-s/2}\mathcal{L}_s(v^{-s/2}f) $ satisfies $ T_su_{0,\delta} = C_{-s,\delta} u_{0,\delta}.$ By Schur test it follows that $ T_s $ is bounded on $ L^2(\He^n) $ and one has the inequality
$$
\Big| \int_{\He^n} T_sf(x) \overline{f(x)} dx \Big| \leq C_{-s,\delta} \int_{\He^n} |f(x)|^2 dx.$$
This inequality is equivalent to the boundedness of $ \mathcal{L}_{-s}^{1/2} v^{-s/2} $ and  $ v^{-s/2}\mathcal{L}_{-s}^{1/2} $ leading to
$$   \int_{\He^n} |v^{-s/2}(x)\mathcal{L}_{-s}^{1/2}f(x)|^2 dx \leq C_{-s,\delta} \int_{\He^n} |f(x)|^2 dx.$$
Applying this to $ \mathcal{L}_s^{1/2}f $ and noting that $ \mathcal{L}_{-s} = \mathcal{L}_s^{-1} $ we obtain the inequality in Theorem \ref{thm:HardynonH} on a dense subspace.
We also observe that if we take $ f = u_{-s,\delta} $ in Theorem \ref{thm:HardynonH} both sides of the inequality reduce to
$$  (4\delta)^s \frac{\Gamma(\frac{Q+2s}{4})^2}{\Gamma(\frac{Q-2s}{4})^2} \int_{\He^n} u_{-s,\delta}(x) u_{s,\delta}(x) \, dx.$$
This proves the optimality of the constant $ C_{s,\delta} $ in our inequality, which is achieved when $ f = u_{-s,\delta}$.

We now proceed to obtain a ground state representation for $ \mathcal{L}_s $ which will lead to the ground state representation (Theorem \ref{thm:gsr}) and hence another proof of Theorem \ref{thm:HardynonH}.
We begin with the next lemma which easily follows from the integral representation proved in Proposition \ref{prop:integralNONHomo}.

\begin{lem}
\label{lem:1}
Let $n\ge1$ and $0<s<1$. Then, for all $f\in W^{s,2}(\He^n)$
\begin{equation}
\label{eq:intLsNONh}
\langle \mathcal{L}_sf,f \rangle=a_{n,s} \int_{\He^n}\int_{\He^n} \frac{|f(x)-f(y)|^2}{ |y^{-1}x|^{Q+2s}} \,dx\,dy,
\end{equation}
where $ a_{n,s} $ is the positive constant
\begin{equation}
\label{eq:ansNONh}
a_{n,s}
=\frac{2^{n-2+3s}}{\pi^{n+1}}\frac{\Gamma\big(\frac{n+1+s}{2}\big)^2}{|\Gamma(-s)|}.
\end{equation}
\end{lem}

\begin{proof}
Let $f\in  C_0^\infty(\He^n)$.
The integral representation obtained in Proposition \ref{prop:integralNONHomo} gives
$$
\langle \mathcal{L}_sf, f\rangle =\frac{1}{|\Gamma(-s)|}\int_0^\infty \bigg( \int_{\He^n} (f(x)-f\ast \mathcal{K}_t^s(x))\overline{f(x)} \,dx \bigg)t^{-s-1} dt.
$$
By Fubini, the integral can be written as
$$
\frac{1}{|\Gamma(-s)|}\int_{\He^n} \int_{\He^n}\big(f(x)-f(y)\big) \overline{f(x)}\mathcal{K}_s(y^{-1}x)\,dx\,dy,
$$
where $\mathcal{K}_s$ is given in Proposition \ref{lem:kernelnonHomoExplicit}.
As the  kernel $\mathcal{K}_s(x)$ is symmetric, i.e., $ \mathcal{K}_s(x) = \mathcal{K}_s(x^{-1}) $ the above is also equal to
$$
\langle \mathcal{L}_sf, f\rangle =\frac{1}{|\Gamma(-s)|}\int_{\He^n} \int_{\He^n}\big(f(y)-f(x)\big) \overline{f(y)}\mathcal{K}_s(y^{-1}x)\,dx\,dy.
$$
Adding them up we get
$$
\langle \mathcal{L}_sf, f\rangle =\frac{1}{2|\Gamma(-s)|} \int_{\He^n} \int_{\He^n} |f(x)-f(y)|^2 \mathcal{K}_s(y^{-1}x)\,dx\,dy.
$$
The justification of the change of order of integration is as follows. By Proposition \ref{lem:kernelnonHomoExplicit}, we have that $ \mathcal{K}_s(x) = c_{n,s} |x|^{-Q-2s} $ with $c_{n,s}$ as in \eqref{eq:cnsNONh}, and we can check that
$$
\int_{\He^n}\int_{\He^n} \frac{|f(x)-f(y)|^2}{ |y^{-1}x|^{Q+2s}} \,dx\,dy < \infty
$$
when $ f \in C_0^\infty(\He^n)$. Consequently, we can apply Fubini to obtain \eqref{eq:intLsNONh} for $ f \in C_0^\infty(\He^n)$.

Let us take now $ f \in W^{s,2}(\He^n)$. Choose a sequence $ f_k \in C_0^\infty(\He^n) $ such that $ f_k $ converges to $f $ in $ W^{s,2}(\He^n)$.
It is clear that $ \langle \mathcal{L}_{s}f_k,f_k \rangle $ converges to $ \langle \mathcal{L}_{s}f,f \rangle $ as $ k $ tends to infinity. Moreover, since we have just proved the result for functions in $C_0^\infty(\He^n) $, we have
\begin{equation}
\label{eq:CauchyNONh}
 \langle \mathcal{L}_{s}f_k,f_k \rangle= a_{n,s}\int_{\He^n}\int_{\He^n} \frac{|f_k(x)-f_k(y)|^2}{ |y^{-1}x|^{Q+2s}} \,dx\,dy<\infty.
\end{equation}
Consequently, the functions $ F_k(x,y) = f_k(x)-f_k(y) $ form a Cauchy sequence in $ L^2(\He^n \times \He^n, d\mu) $ where
$$
d\mu(x,y) = \frac{1}{ |y^{-1}x|^{Q+2s}}  \,dx\,dy
$$
which converges to $ f(x)-f(y) $ in this norm. Hence, passing to the limit in \eqref{eq:CauchyNONh}, we complete the proof of the lemma.

\end{proof}

We are now ready to state the ground state representation for the operator $ \mathcal{L}_s.$ Let us set
$$
\mathcal{H}_s[ f] =  \langle \mathcal{L}_sf, f\rangle - C_{s,\delta} \int_{\He^n}  \frac{|f(z,w)|^2}{ \big((\delta+\frac14 |z|^2)^2+w^2\big)^s} dz\,dw
$$
where   $ C_{s,\delta} = (4\delta)^s \frac{\Gamma(\frac{Q+2s}{4})^2}{\Gamma(\frac{Q-2s}{4})^2}$ is the constant defiend at the beginning of Section 5. Hardy's inequality follows immediately if we could show that $ \mathcal{H}_s[f] $ is nonnegative.

Recall the definition of the function $u_{s,\delta}(x)$ given in \eqref{eq:usdelta}. Theorem \ref{thm:gsrep} below is just Theorem~\ref{thm:gsr}. We repeat the statement here for easy reading.

\begin{thm}
\label{thm:gsrep}
Let $ 0 < s <1 $ and $ \delta >0.$ If $ F \in C_0^\infty(\He^n) $ and $ G(x) =  F(x)u_{-s,\delta}(x)^{-1} $ then
$$
\mathcal{H}_s[F] = a_{n,s}  \int_{\He^n}\int_{\He^n} \frac{|G(x)-G(y)|^2}{ |y^{-1}x|^{Q+2s}} \,u_{-s,\delta}(x) \,u_{-s,\delta}(y)\,dx\,dy,
$$
where $a_{n,s}$ is the positive constant \eqref{eq:ansNONh}.
\end{thm}
\begin{proof} By polarizing the representation in Lemma \ref{lem:1} we get for any $ f, g \in W^{s,2}(\He^n) $,
\begin{equation}
\label{eq:polarizeNONh}
\langle \mathcal{L}_sf,g \rangle= a_{n,s} \int_{\He^n}\int_{\He^n} \frac{(f(x)-f(y))\overline{(g(x)-g(y))}}{ |y^{-1}x|^{Q+2s}} \,dx\,dy.
\end{equation}
We apply the above formula to $ g(x) = u_{-s,\delta} $ and $ f(x) = |F(x)|^2 g(x)^{-1}.$  We remark  that $ u_{-s,\delta} \in W^{s,2}(\He^n)$. Indeed, in view of Proposition \ref{prop:cs} we know that $ \mathcal{L}_s u_{-s,\delta} $ is a constant multiple of $ u_{s,\delta}.$ As both $ u_{s,\delta} $ and $ u_{-s,\delta} $ are square integrable it follows that $ u_{-s,\delta} \in W^{s,2}(\He^n).$  After simplification, the right hand side of \eqref{eq:polarizeNONh} becomes
$$
a_{n,s} \int_{\He^n}\int_{\He^n} \bigg( |F(x)-F(y)|^2- \bigg|\frac{F(x)}{g(x)}-\frac{F(y)}{g(y)}\bigg|^2g(x)g(y)\bigg) \frac{dx \,dy}{ |y^{-1}x|^{Q+2s}}.
$$
On the other hand, in view of Theorem \ref{thm:CH1} the left hand side of \eqref{eq:polarizeNONh} becomes
$$  (4\delta)^s \frac{\Gamma(\frac{Q+2s}{4})^2}{\Gamma(\frac{Q-2s}{4})^2} \int_{\He^n} \frac{|F(x)|^2}{u_{-s,\delta}(x)} u_{s,\delta}(x) \, dx.$$
Since
$$ \frac{u_{s,\delta}(x)}{u_{-s,\delta}(x)} =  \big((\delta+\frac14 |z|^2)^2+w^2\big)^{-s} $$ by recalling the definition of $ G $ and using Lemma \ref{lem:1} we complete the proof of the theorem.

\end{proof}

\subsection{Proof of Theorem \ref{thm:HardyH}}
%%%%%%%%%%%%%%%%%%%%%%%%%%%%%%%%%%%%%%%%%%%%%%%%

We need the following analogue of Lemma \ref{lem:1} which easily follows from the integral representation proved in Proposition \ref{prop:integralHomo}.

\begin{lem}
\label{lem:2}
Let $n\ge1$ and $0<s<1$. Then, for all $f\in W^{1-s,2}(\He^n)$
\begin{equation}
\label{eq:intLs}
\langle \Lambda_{1-s}f,f \rangle=b_{n,s} \int_{\He^n}\int_{\He^n} \frac{|f(x)-f(y)|^2}{ |y^{-1}x|^{Q+2(1-s)}} \omega(y^{-1}x) \,dx\,dy,
\end{equation}
where $ \omega(z,w) = |z|^2 |(z,w)|^{-2}$ and $ b_{n,s} $ is the positive constant
\begin{equation}
\label{eq:bnsHomog}
b_{n,s}=\frac{2^{n+4-3s}}{\pi^{n+1}}
\frac{\Gamma\big(\frac{n+3-s}{2}\big)\Gamma\big(\frac{n+1-s}{2}\big)}{|\Gamma(s-1)|}.
\end{equation}
\end{lem}

\begin{proof} We first assume that $ f \in C_0^\infty(\He^n).$ The integral representation obtained in Proposition \ref{prop:integralHomo} gives
$$
\langle \Lambda_{1-s}f, f\rangle =\frac{1}{|\Gamma(s-1)|}\int_0^\infty \bigg( \int_{\He^n} (f(x)-f\ast {K}_t^s(x))\overline{f(x)} \,dx \bigg)t^{s-2} dt.
$$
As in the proof of Lemma \ref{lem:1}, by using Fubini, the integral can be written as
$$
\frac{1}{|\Gamma(s-1)|} \int_{\He^n} \int_{\He^n}\big(f(x)-f(y)\big) \overline{f(x)}{K}_s(y^{-1}x)\,dx\,dy.
$$
As the  kernel $ {K}_s(x)$ is symmetric, i.e., $ {K}_s(x) = {K}_s(x^{-1}) $ the above is also equal to
$$
\langle \Lambda_{1-s}f, f\rangle =\frac{1}{|\Gamma(s-1)|} \int_{\He^n} \int_{\He^n}\big(f(y)-f(x)\big) \overline{f(y)}{K}_s(y^{-1}x)\,dx\,dy.
$$
Adding them up we get
$$
\langle \Lambda_{1-s}f, f\rangle =\frac{1}{2|\Gamma(s-1)|} \int_{\He^n} \int_{\He^n} |f(x)-f(y)|^2 {K}_s(y^{-1}x)\,dx\,dy.
$$
By Proposition \ref{lem:kernelHomoExplicit}, $ {K}_s(x) = c_{n,s} \omega(x) |x|^{-Q-2(1-s)} $, where $c_{n,s}$ is as in \eqref{eq:cnsHomog}, and we can check that
$$ \int_{\He^n}\int_{\He^n} \frac{|f(x)-f(y)|^2}{ |y^{-1}x|^{Q+2(1-s)}} \omega(y^{-1}x) \,dx\,dy < \infty $$
when $ f \in C_0^\infty(\He^n)$. Consequently, we can apply Fubini to change the order of integration to obtain \eqref{eq:intLs} for $ f \in C_0^\infty(\He^n)$, with $b_{n,s}=\frac{c_{n,s}}{2|\Gamma(s-1)|}$.

We will extend the result to $ f \in W^{1-s,2}(\He^n)$ and, as before, we use a density argument. Choose a sequence $ f_k \in C_0^\infty(\He^n) $ such that $ f_k $ converges to $f $ in $ W^{1-s,2}(\He^n).$
It is clear that $ \langle \Lambda_{1-s}f_k,f_k \rangle $ converges to $ \langle \Lambda_{1-s}f,f \rangle $ as $ k $ tends to infinity. Moreover, as $ \omega $ is bounded function, we have
$$
\int_{\He^n}\int_{\He^n} \frac{|f_k(x)-f_k(y)|^2}{ |y^{-1}x|^{Q+2(1-s)}} \omega(y^{-1}x) \,dx\,dy \leq C \int_{\He^n}\int_{\He^n} \frac{|f_k(x)-f_k(y)|^2}{ |y^{-1}x|^{Q+2(1-s)}} \,dx\,dy = a_{n,1-s} C\langle \mathcal{L}_{1-s}f_k,f_k \rangle
$$
where we have made use of the result in Lemma \ref{lem:1}.
Consequently, the functions $ F_k(x,y) = f_k(x)-f_k(y) $ form a Cauchy sequence in $ L^2(\He^n \times \He^n, d\mu) $ where
$$ d\mu(x,y) = \frac{\omega(y^{-1}x)}{ |y^{-1}x|^{Q+2(1-s)}}  \,dx\,dy$$
which converges to $ f(x)-f(y) $ in this norm. Hence, passing to the limit in
$$
\langle \Lambda_{1-s}f_k,f_k \rangle=b_{n,s} \int_{\He^n}\int_{\He^n} \frac{|f_k(x)-f_k(y)|^2}{ |y^{-1}x|^{Q+2(1-s)}} \omega(y^{-1}x) \,dx\,dy,
$$
we complete the proof of the lemma.
\end{proof}

We are now ready to state the ground state representation for the operator $ \Lambda_{1-s}$. Let us set
$$
H_s[ f] =  \langle \Lambda_{1-s}f, f\rangle - B_{n,s} \int_{\He^n}  \frac{|f(x)|^2}{ |x|^{2(1-s)}} \,dx
$$
where   $B_{n,s} = 2^{2n+3(1-s)} \frac{\Gamma(\frac{n+1-s}{2})^2}{\Gamma(s)\Gamma(\frac{n}{2})^2}$. Hardy's inequality follows immediately if we could show that $ H_s[f] $ is nonnegative. Recall that we have denoted the fundamental solution of $ \mathcal{L}_s $ by $ g_s $,
which is a constant multiple of $ u_{-s,0}$, see \eqref{eq:fundSolution} in Subsection \ref{sec:fundamental}.

\begin{thm}
\label{thm:gsrhomo}
Let $ 0 < s <1$.  Let $ F \in C_0^\infty(\He^n) $ be supported away from $ 0 $ and $ G(x) =  F(x)g_1(x)^{-1} $ then
$$
H_s[F] = b_{n,s}  \int_{\He^n}\int_{\He^n} \frac{|G(x)-G(y)|^2}{ |y^{-1}x|^{Q+2(1-s)}} \,g_1(x) \,g_1(y)\,dx\,dy,
$$
where $b_{n,s}$ is the positive constant \eqref{eq:bnsHomog}.
\end{thm}
\begin{proof}

By polarizing the representation in Lemma \ref{lem:2} we get for any $ f, g \in W^{1-s,2}(\He^n) $,
\begin{equation}
\label{eq:polarizeHomog}
\langle \Lambda_{1-s}f,g \rangle= b_{n,s} \int_{\He^n}\int_{\He^n} \frac{(f(x)-f(y))\overline{(g(x)-g(y))}}{ |y^{-1}x|^{Q+2(1-s)}} \omega(y^{-1}x)\,dx\,dy.
\end{equation}
We apply the above formula to $ g(x) = u_{-1,\delta}(x) $ and $ f(x) = |F(x)|^2 g(x)^{-1}.$ After simplification, the right hand side of \eqref{eq:polarizeHomog} becomes
$$
 b_{n,s} \int_{\He^n}\int_{\He^n} \bigg( |F(x)-F(y)|^2- \bigg|\frac{F(x)}{u_{-1,\delta}(x)}-\frac{F(y)}{u_{-1,\delta}(y)}\bigg|^2
 u_{-1,\delta}(x)u_{-1,\delta}(y)\bigg) \frac{\omega(y^{-1}x)}{ |y^{-1}x|^{Q+2(1-s)}} \,dx \,dy.
$$
On the other hand, the left hand side of \eqref{eq:polarizeHomog} can be simplified using the explicit formula for the Fourier transform of $ u_{-1,\delta}$.  Observe that
$  \langle \Lambda_{1-s}f,g \rangle = \langle f, \mathcal{L}_s^{-1}\mathcal{L}u_{-1,\delta} \rangle = \langle f,v_{s,\delta} \rangle $
where $ v_{s,\delta}(x) = \mathcal{L}_s^{-1}\mathcal{L}u_{-1,\delta}(x)$. Thus we have the identity
\begin{equation}
\label{eq:gs1}
\langle f,v_{s,\delta} \rangle =  b_{n,s} \int_{\He^n}\int_{\He^n} \bigg( |F(x)-F(y)|^2- \bigg|\frac{F(x)}{u_{-1,\delta}(x)}
-\frac{F(y)}{u_{-1,\delta}(y)}\bigg|^2u_{-1,\delta}(x)u_{-1,\delta}(y)\bigg) \frac{\omega(y^{-1}x)}{ |y^{-1}x|^{Q+2(1-s)}} \,dx \,dy.
\end{equation}
By the arguments showed in Sections \ref{sec:prelim} and \ref{sec:fundamental} we can deduce that the Fourier transform of $ v_{s,\delta} $ is given by
$$
\widehat{v_{s,\delta}}(\lambda) = \sum_{k=0}^\infty a_{k,\delta}^\lambda(s) P_k(\lambda)
$$
where
$$
a_{k,\delta}^\lambda(s) = (2k+n)|\lambda| c_{k,\delta}^\lambda(-1) (2|\lambda|)^{-s} \frac{\Gamma(\frac{2k+n}{2}+\frac{1-s}{2})}{\Gamma(\frac{2k+n}{2}+\frac{1+s}{2})}.
$$
Using the explicit formula for $ c_{k,\delta}^\lambda (s) $ in Proposition \ref{prop:Ftransform} we have
$$
a_{k,\delta}^\lambda(s)=(2k+ n)|\lambda|  (2|\lambda|)^{-s} \frac{\Gamma(\frac{2k+n}{2}+\frac{1-s}{2})}{\Gamma(\frac{2k+n}{2}+\frac{1+s}{2})} \frac{(2\pi)^{n+1}|\lambda|^{-1}}{\big(\Gamma(\frac{n}{2})\big)^2} L\Big(\delta \lambda, \frac{2k+n}{2}, \frac{2k+n}{2}+1\Big).
$$
By letting $ \delta $ go to zero and noting that
$$
L\Big(0, \frac{2k+n}{2}, \frac{2k+n}{2}+1\Big) = \frac{\Gamma(\frac{2k+n}{2})}{\Gamma(\frac{2k+n}{2}+1)} = \frac{2}{2k+n}
$$
we see that $ v_{s,\delta}$ converges in the sense of distributions to
$$
\frac{2 (2\pi)^{n+1}}{\Gamma(\frac{n}{2})^2} g_s(x) = \frac{2 (2\pi)^{n+1}}{\Gamma(\frac{n}{2})^2} \frac{2^{n+1-3s}\Gamma\big(\frac{n+1-s}{2}\big)^2}{\pi^{n+1}\Gamma(s)}|x|^{-Q+2s}.
$$
Thus $ \langle f, v_{s,\delta} \rangle $ converges to
$$
\frac{2^{2n+3(1-s)}\Gamma\big(\frac{n+1-s}{2}\big)^2}{\Gamma(s)\Gamma(\frac{n}{2})^2} \int_{\He^n}  \frac{|F(x)|^2}{|x|^{2(1-s)}}\,dx.
$$
On the other hand, as $ F $ is supported away from $ 0 $, the right hand side of \eqref{eq:gs1} converges to
$$ \langle f,v_{s,\delta} \rangle =  b_{n,s} \int_{\He^n}\int_{\He^n} \bigg( |F(x)-F(y)|^2- \bigg|\frac{F(x)}{g_1(x)}-\frac{F(y)}{g_1(y)}\bigg|^2 g_1(x)g_1(y)\bigg) \frac{\omega(y^{-1}x)}{ |y^{-1}x|^{Q+2(1-s)}} \,dx \,dy.$$
Since
$$  b_{n,s} \int_{\He^n}\int_{\He^n}  |F(x)-F(y)|^2 \frac{\omega(y^{-1}x)}{ |y^{-1}x|^{Q+2s}} \,dx \,dy = \langle \Lambda_{1-s}F, F\rangle, $$
the ground state representation is proved.

\end{proof}

\begin{rem}
The ground state representation proved above immediately leads to Hardy's inequality under the assumption that $ F $ is supported away from the origin. However, this extra condition can be removed arguing as follows. Note that for  any $ \delta > 0 $ we have proved the inequality
$$  \int_{\He^n} \frac{|F(x)|^2}{u_{-1,\delta}(x)} v_{s,\delta}(x) \,dx \leq \langle \Lambda_{1-s}F,F\rangle $$
valid for any $ F \in C_0^\infty(\He^n).$ Since
$$
4^{-n} \int_{\He^n} |F(x)|^2 |x|^{2n} v_{s,\delta}(x) \,dx \leq \int_{\He^n} \frac{|F(x)|^2}{u_{-1,\delta}(x)} v_{s,\delta}(x) \,dx \leq \langle \Lambda_{1-s}F,F\rangle
$$
we can pass to the limit as $ \delta $ goes to zero. As $ v_{s,\delta} $ converges in the sense of distributions to a constant multiple of $ g_s$ we get the required inequality.
\end{rem}

%%%%%%%%%%%%%%%%%%%%%%%%%%%%
\subsection{Hardy inequalities for $ \mathcal{L}^s$}
\label{sub:Ls}
%%%%%%%%%%%%%%%%%%%%%%%%%%%%

Now, by comparing $\mathcal{L}_s$ with $\mathcal{L}^{s}$ we can obtain Hardy  inequalities for $\mathcal{L}^{s}$. We have stated them in Theorem \ref{thm:2} and Theorem \ref{thm:4}. These inequalities involve certain bounded operators $ U_s$ and $ V_s$ and by estimating the norms of these, we can get Hardy inequalities for $ \mathcal{L}^s$. Though the resulting inequalities are not sharp, we state them here for the sake of completeness.

We will estimate in detail the norm of $V_s$. Since $ V_s =  \mathcal{L}_{1-s}^{-1}\mathcal{L}\mathcal{L}^{-s},$ it corresponds to the multiplier
$$
\Big(\frac{2k+n}{2}\Big)^{1-s}\frac{\Gamma\big(\frac{2k+n}{2}+\frac{s}{2}\big)}
{\big(\frac{2k+n}{2}+\frac{2-s}{2}\big)}
$$
which clearly shows that  it is bounded on $L^2(\He^n)$. Moreover, the formula (see for instance \cite[Section 7]{TE}),
$$
\frac{\Gamma(x+\alpha)}{\Gamma(x+\beta)}=\frac{1}{\Gamma(\beta-\alpha)}
\int_0^{\infty}e^{-(x+\alpha)v}(1-e^{-v})^{\beta-\alpha-1}\,dv
$$
valid for $\beta-\alpha>0$, gives
\begin{align*}
(x+\beta)\frac{\Gamma(x+\alpha)}{\Gamma(x+\beta+1)}&=\frac{(x+\beta)}{\Gamma(\beta+1-\alpha)}
\int_0^{\infty}e^{-(x+\alpha)v}(1-e^{-v})^{\beta-\alpha}\,dv\\
&\le (x+\beta) (x+\alpha)^{-(\beta-\alpha)-1}.
\end{align*}
Consequently, if $\alpha>0$, we have that $ x^{\beta-\alpha}\frac{\Gamma(x+\alpha)}{\Gamma(x+\beta)}\le \frac{x+\beta}{x+\alpha}$. With $x=\frac{2k+n}{2}$, $\beta=\frac{2-s}{2}$ and $\alpha=\frac{s}{2}$ we get
$$
\Big(\frac{2k+n}{2}\Big)^{1-s}\frac{\Gamma\big(\frac{2k+n}{2}+\frac{s}{2}\big)}
{\big(\frac{2k+n}{2}+\frac{2-s}{2}\big)}\le \frac{(2k+n+2-s)}{(2k+n+s)}\le \frac{(n+2-s)}{(n+s)}.
$$
Thus, we have the inequality
$$
\langle \mathcal{L}^{s}f,f \rangle\ge\frac{(n+s)}{(n+2-s)}\frac{2^{2n+3s}\Gamma\big(\frac{n+s}{2}\big)^2}
{\Gamma(1-s)\Gamma\big(\frac{n}{2}\big)^2}\int_{\He^n}\frac{|f(x)|^2}{|x|^{2s}}\,dx.
$$
In a similar way we can also estimate the norm of $ U_s = \mathcal{L}_s \mathcal{L}^{-s}$, that is given by \eqref{eq:Us}. We leave the computation for the interested reader.

%%%%%%%%%%%%%%%%%%%%%%%%%%%%%%%%%%%%
\subsection{Hardy--Littlewood--Sobolev inequality for $ \mathcal{L}_s$}
\label{sub:HLS}
%%%%%%%%%%%%%%%%%%%%%%%%%%%%%%%%%%%%

In this subsection we briefly recall the Hardy--Littlewood--Sobolev inequality for $ \mathcal{L}_s $ due to Frank and Lieb in \cite{FL} and show how to deduce a slightly weaker form of Hardy inequality for $ \mathcal{L}_s$. In the present subsection we follow the notation used in \cite{BFM}. Therein the group law on the Heisenberg group is given by
$$
(z,w)(z',w') = (z+z', w+w'+ 2\Im(z\cdot \bar{z'})
$$
and the sublaplacian $ L $ is defined as
$$ L = -\frac14 \sum_{j=1}^n \big( \tilde{X}_j^2+\tilde{Y}_j^2 \big).$$
Here the vector fields adapted to the above group structure are given by
$$
\tilde{X}_j=\bigg(\frac{\partial}{\partial x_j}+2 y_j\frac{\partial }{\partial t}\bigg), \quad \tilde{Y}_j=\bigg(\frac{\partial}{\partial y_j}-2 x_j\frac{\partial}{\partial t}\bigg),  \quad j=1,2,\ldots,n.
$$
Recall that our sublaplacian $\mathcal{L}$ is defined by
$$
\mathcal{L}=- \sum_{j=1}^n(X_j^2+Y_j^2)$$
with $$ X_j=\bigg(\frac{\partial}{\partial x_j}+\frac12 y_j\frac{\partial}{\partial t}\bigg), \quad Y_j=\bigg(\frac{\partial}{\partial y_j}-\frac12 x_j\frac{\partial}{\partial t}\bigg), \quad j=1,2,\ldots,n.
$$
It is easy to see that these two operators are related by the equation  $ Lg(z,w) = (\mathcal{L}f)(2z,w) $, where $ g(z,w) = f(2z,w)$. More generally, we have
$ L_sg(z,w) = (\mathcal{L}_sf)(2z,w)$.

The HLS inequality as stated in \cite{BFM} reads as
$$
\frac{\Gamma\big(\frac{1+n+s}{2}\big)^2}
{\Gamma\big(\frac{1+n-s}{2}\big)^2}  \omega_{2n+1}^{\frac{s}{n+1}}  \bigg(\int_{\He^n} |g(z,w)|^{\frac{2(n+1)}{n+1-s}} dz dw \bigg)^{\frac{n+1-s}{(n+1)}}
\leq \langle {L}_sg,g\rangle.
$$
Since
$$
\int_{\He^n} L_sg(z,w) \overline{g(z,w)} dz\,dw = 2^{-2n} \int_{\He^n} \mathcal{L}_sf(z,w) \overline{f(z,w)} dz\,dw
$$
the HLS inequality for $ \mathcal{L}_s $ takes the form
$$
\frac{\Gamma\big(\frac{1+n+s}{2}\big)^2}
{\Gamma\big(\frac{1+n-s}{2}\big)^2}  \omega_{2n+1}^{\frac{s}{n+1}}  \bigg( 2^{-2n} \int_{\He^n} |f(z,w)|^{\frac{2(n+1)}{n+1-s}} dz dw \bigg)^{\frac{n+1-s}{(n+1)}}
\leq 2^{-2n}  \langle {L}_sf,f\rangle
$$
Now by applying Holder's inequality,
$$
\int_{\He^n}\frac{|g(z,w)|^2}{\big( (1+|z|^2)^2+w^2\big)^{s}}\,dz\,dw \leq
 k(n,1)^{\frac{s}{n+1}}\bigg(  \int_{\He^n} |g(z,w)|^{\frac{2(n+1)}{n+1-s}} dz dw \bigg)^{\frac{n+1-s}{(n+1)}}
$$
where $ k(n,1) $ is the constant defined by
$$
k(n,1) = \int_{\He^n} \big( (1+|z|^2)^2+w^2\big)^{-n-1} dz dw.
$$
This integral has been evaluated in \cite{CH} and we have
$$
k(n,1) =   \frac{\pi^{n+1} 2^{-2n}}{\Gamma(n+1)} =  2^{-2n-1}\omega_{2n+1}
$$
where the measure $ \omega_{n} $ of the unit sphere $ \mathbb{S}^n $ in $ \R^{n+1} $ is given by $ \frac{2 \pi^{\frac{n+1}{2}}}{\Gamma(\frac{n+1}{2})}$.
Thus we have
\begin{multline*}
 4^s \frac{\Gamma\big(\frac{1+n+s}{2}\big)^2}
{\Gamma\big(\frac{1+n-s}{2}\big)^2}\int_{\He^n}\frac{|g(z,w)|^2}{\big( (1+|z|^2)^2+w^2\big)^{s}}\,dz\,dw \\
\leq 4^s  \frac{\Gamma\big(\frac{1+n+s}{2}\big)^2}
{\Gamma\big(\frac{1+n-s}{2}\big)^2} \big(2^{-2n-1}\omega_{2n+1}\big)^{\frac{s}{n+1}} \bigg( \int_{\He^n} |g(z,w)|^{\frac{2(n+1)}{n+1-s}} dz dw \bigg)^{\frac{n+1-s}{(n+1)}}.
\end{multline*}
In view of the HLS inequality, we obtain
$$
 4^s \frac{\Gamma\big(\frac{1+n+s}{2}\big)^2}
{\Gamma\big(\frac{1+n-s}{2}\big)^2}\int_{\He^n}\frac{|g(z,w)|^2}{\big( (1+|z|^2)^2+w^2\big)^{s}}\,dz\,dw \leq 2^{\frac{s}{n+1}} \langle L_sg,g\rangle.
$$
Consequently, we have the inequality
$$
 4^s \frac{\Gamma\big(\frac{1+n+s}{2}\big)^2}
{\Gamma\big(\frac{1+n-s}{2}\big)^2}\int_{\He^n}\frac{|f(z,w)|^2}{\big( (1+\frac14|z|^2)^2+w^2\big)^{s}}\,dz\,dw \leq 2^{\frac{s}{n+1}}  \langle \mathcal{L}_sf,f\rangle
$$
which is weaker than the inequality stated in Theorem \ref{thm:HardynonH}.

\appendix

%%%%%%%%%%%%%%%%%%%%%%%%%%%%%%%%%%%%%%%%%%%%%%%%%%%%%%%%%%%%%%%%%%
\section{Hardy's inequality in the Euclidean case revisited}
%%%%%%%%%%%%%%%%%%%%%%%%%%%%%%%%%%%%%%%%%%%%%%%%%%%%%%%%%%%%%%%%%%

For the sake of completeness, here we recall and reprove the fractional Hardy inequality in the Euclidean space \eqref{eq:HardyHomoEucl}.

It was already said in the introduction that an improvement of this inequality was obtained by Frank et al in \cite{FLS}, by using the ground state representation technique. We are going to reproduce the proof, but getting an integral representation of the fractional powers of the Euclidean Laplacian via the semigroup language. Although the integral representation in this Euclidean case is well known, maybe the use of the semigroup language to get it is not so known. Moreover, since our integral representations for the conformally invariant powers of the sublaplacian are based on the semigroup language, we would also like to show that the Euclidean case can be treated with the semigroup approach as well. Moreover, the constants are quickly obtained in this way.

We recall that the heat semigroup was introduced systematically to define fractional powers of second order partial differential operators in \cite{ST}.

Let us follow the scheme we showed for the fractional powers of the conformally invariant sublaplacian on $\He^n$. Nevertheless, we will not show the proofs rigorously, since they can be found somewhere else.

For $x\in \R^n$ and $t>0$, let $G_t(x)$ denote the Euclidean heat kernel, that is,
$$
G_t(x)=\frac{1}{(4\pi t)^{n/2}}e^{-\frac{|x|^2}{4t}}.
$$
For a function $f$ good enough, the \textit{heat semigroup} $e^{-t\Delta}f$ is defined as the convolution $(G_t\ast f)(x)$, thus
$$
e^{-t\Delta}f(x)=\int_{\R^n}G_t(x-y)f(y)\,dy.
$$
It is very well known that $e^{-t\Delta} 1=1$.

Let $0<s<1$. In terms of $G_t$, we define another kernel $\mathcal{G}_s$ by
$$
\mathcal{G}_s(x)=\frac{1}{|\Gamma(-s)|}\int_0^{\infty}G_t(x)t^{-s-1}\,dt.
$$
This kernel can be explicitly computed. Actually, we have a more general result. Let $\alpha \in \R$ be such that $ 0 <\alpha<n/2$, and define
$$
g_\alpha(x):=\frac{1}{\Gamma(\alpha)}\frac{1}{(4\pi)^{n/2}}\int_0^{\infty}e^{-\frac{|x|^2}{4t}}t^{\alpha-1-n/2}\,dt.
$$
We can easily check the following elementary lemma.

\begin{lem}
\label{lem:computation}
Let  $n\ge 1$ and $\alpha \in \R$ be such that $ 0 < \alpha<n/2$ and $x\in\R^n$. Then,
$$
g_\alpha(x)=\frac{\Gamma(n/2-\alpha)}{\Gamma(\alpha)4^{\alpha}\pi^{n/2}}|x|^{2\alpha-n}.
$$
\end{lem}

\begin{proof}
The proof follows immediately after making the change of variable $z=\frac{|x|^2}{4t}$, and taking into account the definition of the Gamma function.
\end{proof}
Observe that, in particular, from Lemma \ref{lem:computation}, we have that
\begin{equation}
\label{eq:euclideanKernels}
\mathcal{G}_s(x)=\frac{4^{s}\Gamma(n/2+s)}{|\Gamma(-s)|\pi^{n/2}}|x|^{-2s-n}.
\end{equation}

An integral representation for $\Delta^s f$ can be obtained for functions $f\in \mathcal{S}$, where $\mathcal{S}$ is the class of rapidly decreasing $C^{\infty}(\R^n)$ functions. The complete, rigorous proof of this result can be found in \cite[Lemma 5.1]{ST}.

\begin{prop}
\label{prop:pointwiseEuc}
Let $n\ge 1$ and $0<s<1$. Then, for all $f\in \mathcal{S}$, we have the following pointwise representation
$$
\Delta^s f(x)=\operatorname{P.V.}\int_{\R^n}\big(f(x)-f(y)\big)\mathcal{G}_s(x-y)\,dy,
$$
where $\mathcal{G}_s(x)$ is given in \eqref{eq:euclideanKernels}.
\end{prop}
\begin{proof}
We have
\begin{align*}
e^{-t\Delta}f(x)-f(x)=e^{-t\Delta}f(x)-f(x)e^{-t\Delta}1(x)&=\int_{\R^n}G_t(x-y)f(y)\,dy-f(x)\int_{\R^n}G_t(x-y)\,dy\\
&=\int_{\R^n}G_t(x-y)\big(f(y)-f(x)\big)\,dy.
\end{align*}
Then, motivated by the numerical identity $\lambda^s=\frac{1}{\Gamma(-s)}\int_0^{\infty}(e^{-t\lambda}-1)\frac{dt}{t^{1+s}}$, $\lambda>0$, we have
\begin{align*}
\Delta^s f(x)&=\frac{1}{\Gamma(-s)}\int_0^{\infty}\big(e^{-t\Delta}f(x)-f(x)\big)\frac{dt}{t^{1+s}}\\
&=\frac{1}{\Gamma(-s)}\int_0^{\infty}\int_{\R^n}G_t(x-y)\big(f(y)-f(x)\big)\,dy\frac{dt}{t^{1+s}}\\
&=\frac{1}{\Gamma(-s)}\int_{\R^n}\big(f(y)-f(x)\big)\int_0^{\infty}G_t(x-y)\frac{dt}{t^{1+s}}\,dy\\
&=\int_{\R^n}\big(f(x)-f(y)\big)\mathcal{G}_s(x-y)\,dy.
\end{align*}
The justification of the change of the order of integration is detailed in \cite[Lemma 3.1]{ST}
\end{proof}
Then, the procedure is as described in \cite{FLS}. The next lemma follows from the integral representation in Proposition \ref{prop:pointwiseEuc}, by using the symmetry of the kernel.

\begin{lem}
\label{lem:repEuc}
Let $n\ge1$ and $0<s<1$. Then, for all $f\in C^{\infty}_0(\R^n)$
$$
\langle \Delta^s f,f\rangle=e_{n,s}\int_{\R^n}\int_{\R^n}\frac{|f(x)-f(y)|^2}{|x-y|^{n+2s}}\,dx\,dy,
$$
where $e_{n,s}$ is the positive constant
$$
e_{n,s}=\frac{4^{s}\Gamma(n/2+s)}{2|\Gamma(-s)|\pi^{n/2}}.
$$
\end{lem}
Finally, let the corresponding ground state representation for the operator $\Delta^s$ be given by
$$
\mathrm{H}_s[f]=\langle \Delta^sf,f\rangle-E_{n,s}\int_{\He^n}\frac{|f(x)|^2}{|x|^{2s}}\,dx,
$$
where
\begin{equation}
\label{eq:sharpEuc}
E_{n,s}=4^s\frac{\Gamma(\frac{n+2s}{4})^2}{\Gamma(\frac{n-2s}{4})^2}.
\end{equation}
In the following theorem it is shown that $\mathrm{H}_s[f]$ is positive.
\begin{thm}
\label{thm:groundPositiveEuc}
Let $0<s<1$, $s<n/2$, and $\alpha>s$. If $u\in C^{\infty}_0(\R^n)$ and $v(x)=u(x)(g_{\alpha}(x))^{-1}$. Then
$$
\mathrm{H}_s[u]=e_{n,s}\int_{\R^n}\int_{\R^n}\frac{|v(x)-v(y)|^2}{|x-y|^{n+2s}}g_{\alpha}(x)g_{\alpha}(y)\,dx\,dy.
$$
\end{thm}
\begin{proof}
By polarizing the representation in Lemma \ref{lem:repEuc}, we get, for any $f,g\in C^{\infty}_0(\R^n)$,
\begin{equation}
\label{eq:polarizeEuc}
\langle \Delta^sf,g\rangle=e_{n,s}\int_{\R^n}\int_{\R^n}\frac{(f(x)-f(y))(g(x)-g(y))}{|x-y|^{n+2s}}\,dx\,dy.
\end{equation}
We take $g(x)=g_{\alpha}(x)$ and $f(x)=|u(x)|^2g_{\alpha}(x)^{-1}$.
Since $\widehat{g}_{\alpha}(\xi)=|\xi|^{-\alpha}$, and by Plancherel, the left hand side of \eqref{eq:polarizeEuc} equals
$$
\int_{\R^n}|\xi|^s\widehat{f}(\xi)\widehat{g}(\xi)\,d\xi=\int_{\R^n}\widehat{f}(\xi)|\xi|^{s-\alpha}\,d\xi=\int_{\R^n}|u(x)|^2\frac{g_{\alpha-s}(x)}{g_{\alpha}(x)}\,d x.
$$
After simplification, the right hand side of \eqref{eq:polarizeEuc} becomes
$$
 e_{n,s} \int_{\R^n}\int_{\R^n} \bigg( |u(x)-u(y)|^2- \bigg|\frac{u(x)}{g_{\alpha}(x)}-\frac{u(y)}{g_{\alpha}(y)}\bigg|^2
g_{\alpha}(x)g_{\alpha}(y)\bigg) \frac{1}{ |x-y|^{n+2s}} \,dx \,dy.
 $$
 By Lemma \ref{lem:repEuc}, and taking into account the definition of $g_\alpha$, the proof is completed.
\end{proof}

As a corollary, we recover the fractional Hardy inequality in the Euclidean space.
\begin{cor}
Let $n\ge1$ and $0<s<1$ such that $n/2>s$. Then, for $u\in C^{\infty}_0(\R^n)$, we have
$$
E_{n,s}\int_{\R^n} \frac{|f(x)|^2}{|x|^{2s}} \, dx \leq   \langle \Delta^{s}f, f\rangle,
$$
where the sharp constant $E_{n,s}$ is given in \eqref{eq:sharpEuc}.
\end{cor}
\begin{proof}
From Lemma \ref{lem:repEuc} and Theorem \ref{thm:groundPositiveEuc} we can deduce immediately that
$$
\langle \Delta^{s}f, f\rangle\ge \int_{\R^n}|f(x)|^2\frac{g_{\alpha-s}(x)}{g_{\alpha}(x)}\,dx=\frac{4^s\Gamma(n/2-\alpha+s)\Gamma(\alpha)}{\Gamma(\alpha-s)\Gamma(n/2-\alpha)} \int_{\R^n}|f(x)|^2\frac{dx}{|x|^{2s}}
$$
where we used Lemma \ref{lem:computation} in the last equality. By choosing $\alpha=\frac{n}{4}+\frac{s}{2}$, we obtain the required result.
\end{proof}
%%%%%%%%%%%%%%%%%
\subsubsection*{Acknowledgments.} The authors are very grateful to the referee for the careful reading of the manuscript and the useful suggestions, especially for indicating a simple proof of Theorem \ref{thm:HardynonH}.
The work leading to this article began in Bangalore when the first author visited the Indian Institute of Science and was completed in Logro\~{n}o when the second author visited the Departamento de Matem\'aticas y Computaci\'on at Universidad de La Rioja, Spain. They wish to thank both institutes for the kind hospitality shown to them during  their visits. The second author is immensely grateful to the Roncal family for the affection he received while he stayed in Logro\~{n}o.

%%%%%%%%%%%%%%%%%%%%%%%%%%%%%%%%%%%%%%%%%%%%%%%%%%%%%%

\end{document}